\documentclass[final]{siamltex}
\usepackage{amssymb, amsmath}
\usepackage{float,epsfig}
\usepackage{algpseudocode}
\usepackage{algorithm}
\usepackage{mathrsfs}
\usepackage{color}
\usepackage{graphicx}
\newtheorem{remark}[theorem]{ Remark}

	\newcommand{\be}{\begin{equation}}
	\newcommand{\ee}{\end{equation}}

\newcommand{\beano}{\begin{eqnarray*}}
	\newcommand{\eeano}{\end{eqnarray*}}

\newcommand{\ba}{\begin{array}}
	\newcommand{\ea}{\end{array}}
\newcommand{\von}{\vskip 1ex}
\newcommand{\vone}{\vskip 2ex}

\def\bmatrix#1{\left[ \begin{matrix} #1 \end{matrix} \right]}

\def \noin{\noindent}
\def \eig{\mathrm{eig}}

\def \rank{\mathrm{rank}}
\def \nrank{\mathrm{nrank}}
\def \sig{\sigma}
\def \lam{\lambda}

\def \C{{\mathbb C}}
\def \diag{\mathrm{diag}}

\def \eip{\mathrm{eip}}

\def \N{\mathbf{N}}
\def \H{\mathbf{H}}

\begin{document}
	\pagestyle{plain}
	
	\title{Analytic matrix descriptions with application to time-delay systems}

	\author{Rafikul alam \thanks{ Corresponding author, Department of Mathematics, IIT Guwahati, Guwahati - 781039, India  ({\tt rafik@iitg.ac.in, rafikul68@gmail.com }) Fax: +91-361-2690762/2582649.  } \and  Jibrail Ali  \thanks{Department of Mathematics, IIT Guwahati, Guwahati - 781039, India ({\tt jibrail@iitg.ac.in, zibrailali@gmail.com}) }   }
	\date{}
	\maketitle

	\maketitle
	
	\begin{abstract}
	A polynomial matrix description(PMD) of a rational matrix $G(\lam)$ is a matrix polynomial of the form $$ \mathbf{P}(\lambda) := \left[\begin{array}{c|c} A(\lambda) & B(\lambda) \\ \hline -C(\lambda) & D(\lambda)\end{array}\right] \text{ such that }  G(\lam) = D(\lambda) + C(\lambda) A(\lambda)^{-1} B(\lambda),$$ where $A(\lam)$ is regular and called the state matrix.  PMDs  have been studied extensively in the context of higher order linear time-invariant systems. An analytic matrix description(AMD) of a meromorphic matrix $ M(\lam)$ is a holomorphic matrix (i.e., a holomorphic matrix-valued function) of the form $$ \mathbf{H}(\lambda) := \left[\begin{array}{c|c} A(\lambda) & B(\lambda) \\ \hline -C(\lambda) & D(\lambda)\end{array}\right] \text{ such that }  M(\lam) = D(\lambda) + C(\lambda) A(\lambda)^{-1} B(\lambda),$$ where $A(\lam)$ is regular and called the state matrix. AMDs arise in the study of linear time-invariant time-delay systems (TDS). Our aim is to develop a framework for analysis of AMDs analogous  to the framework for PMDs and discuss  the extent to  which results of PMDs  can be generalized to  AMDs. We show that important results (e.g., coprime matrix-fraction descriptions (MFDs), least order PMDs, equivalence of PMDs, canonical forms, characterization of PMDs by transfer functions, structural indices of zeros and poles)  which hold for PMDs can be generalized to AMDs which in turn can be utilized to analyze time-delay systems. 
	
	\end{abstract}



\section{Introduction}
	A polynomial matrix description(PMD) of a rational matrix $G(\lam) \in \C(\lam)^{m\times n}$ is an $(r+m)\times (r+n)$ matrix polynomial $\mathbf{P}(\lam)$ of the form \be\label{pmd} \mathbf{P}(\lambda) := \left[\begin{array}{c|c} A(\lambda) & B(\lambda) \\ \hline -C(\lambda) & D(\lambda)\end{array}\right] \text{ such that }  G(\lam) = D(\lambda) + C(\lambda) A(\lambda)^{-1} B(\lambda),\ee where $A(\lam) \in \C[\lam]^{r\times r}$ is a regular matrix polynomial  and called the state matrix.  PMDs  have been studied extensively in the context of higher order linear time-invariant systems; see~\cite{kailath, vardulakis, rosenbrock70}. In fact, $\mathbf{P}(\lam)$ is a system matrix of the linear time-invariant (LTI) system 
	\be \label{lti}	\begin{array}{l} 
				A(\frac{d}{dt})\mathbf{x}(t)  =  B(\frac{d}{dt})\mathbf{u}(t) \\
			\mathbf{y}(t) = C(\frac{d}{dt})  \mathbf{x}(t) + D(\frac{d}{dt}) \mathbf{u}(t)
	\end{array}
	\ee
with state matrix $A(\lam)$ and transfer function $G(\lam)$, where $\mathbf{x}(t) \in \C^r$ is a vector of state variables and $ \mathbf{u}(t) \in \C^n$ is a vector of control variables. In this context, the system matrix $\mathbf{P}(\lam)$ is referred to as Rosenbrock system matrix and the quadruple $[A, B, C, D]$ is referred to as PMD~\cite{ kailath, vardulakis}. We make no such distinction and refer to $\mathbf{P}(\lam)$ is a PMD. The system matrix $\mathbf{P}(\lam)$ and the transfer function $G(\lam)$ play crucial roles in the analysis of the LTI system (\ref{lti}); see~\cite{kailath, vardulakis, rosenbrock70}.

Now, consider  the LTI time-delay system (TDS)~\cite{dde1, dde2,dde3, tdsbook1, tdsbook3, tdsbook2, WN, tds2006} 
\be \label{tds} \begin{array}{l} \frac{\mathrm{d}\mathbf{x}}{\mathrm{d}t} = A_0 \mathbf{x}(t) + \sum^{N_1}_{j=1} A_j  \mathbf{x}(t - \tau_j) + \sum^{N_2}_{j=1} B_j \mathbf{u}(t- t_j), \\ \mathbf{y}(t) = \sum^{N_3}_{j=1} C_j \mathbf{x}(t-s_j) + \sum^{N_4}_{j=1} D_j \mathbf{u}(t-h_j)\end{array}, \ee
where $\mathbf{x}(t) \in \C^r$ is a vector of state variables and $ \mathbf{u}(t) \in \C^n$ is a vector of control variables.  
Further, $ 0< \tau_1< \cdots < \tau_{N_1}$, $ 0 \leq t_1 < \cdots < t_{N_2}$, $ 0\leq s_1 < \cdots < s_{N_3}$ and $ 0\leq h_1 < \cdots < h_{N_4}$ are delay parameters. Furthermore, $A_0, \ldots, A_{N_1}$  are matrices in $\C^{r\times r}$, $B_1, \ldots, B_{N_2}$ are matrices in $\C^{r\times n}$,  $C_1, \ldots, C_{N_3}$ are matrices in  $ \C^{m\times r}$  and $ D_1, \ldots, D_{N_4}$ are matrices $\C^{m\times n}$.

	For the ansatz $ \bmatrix{ \mathbf{x}(t) \\ \mathbf{u}(t)} :=  \bmatrix{ \mathbf{x} \\ -\mathbf{u}} e^{\lam t}$ with $\mathbf{x} \in \C^{r}$ and $ \mathbf{u} \in \C^{n},$ the TDS yields 
 \be \label{tds2}\left[\begin{array}{c|c} A(\lambda) & B(\lambda) \\ \hline -C(\lambda) & D(\lambda)\end{array}\right] \left[\begin{array}{c} \mathbf{x} \\ -\mathbf{u}\end{array}\right] e^{\lambda t} = \left[\begin{array}{c} 0 \\ -\mathbf{y}(t)\end{array}\right] \ee and eliminating $\mathbf{x}(t)$, we have 	 $$ [D(\lambda) + C(\lambda) A(\lambda)^{-1} B(\lambda)] \mathbf{u} e^{\lambda t} =  \mathbf{y}(t),$$ where   $ A(\lambda)  := \lambda I_r - A_0 - \sum^{N_1}_{j=1} A_j e^{-\lambda \tau_j},$ $  B(\lambda) := \sum^{N_2}_{j=1} B_j e^{-\lambda t_j},$ $C(\lambda) := \sum^{N_3}_{j=1} C_j e^{-\lambda s_j}$ and $ D(\lambda) := \sum^{N_4}_{j=1} D_j e^{-\lambda h_j}$ are holomorphic matrices (i.e., holomorphic  matrix-valued functions). 
 	The holomorphic matrix   \be \label{sysmat} \mathbf{H}(\lam) := \left[\begin{array}{c|c} A(\lambda) & B(\lambda) \\ \hline -C(\lambda) & D(\lambda)\end{array}\right]_{(r+m)\times (r+n)}   \ee is a system matrix of the TDS  and the meromorphic matrix (i.e., meromorphic matrix-valued function)  $$M(\lam ) := D(\lam) + C(\lam) A(\lam)^{-1} B(\lam)$$ is the transfer function of the TDS. We refer to $\H(\lam)$ as an analytic matrix description (AMD) and $M(\lam)$ as the transfer  function of $\H(\lam).$

 It follows from (\ref{tds2}) that  $\mathbf{y}(t) = 0 $ whenever $  \H(\lam) \bmatrix{ \mathbf{x} \\ - \mathbf{u} } = 0.$ Alternatively,  $\mathbf{y}(t) = 0 $ whenever $M(\lam) \mathbf{u} = 0,$ that is, when $(\lam, \mathbf{u})$ is an eigenpair of $M(z).$ 
 This shows that the zero output of the TDS can be analyzed via eigenvalues and eigenvectors of the system matrix $\mathbf{H}(\lam)$ as well as eigenvalues and eigenvectors  of the transfer function $M(\lam).$ Thus, as in the case of LTI system (\ref{lti}), a system matrix $\mathbf{H}(\lam)$ and the transfer function $M(\lam)$ are expected to play crucial roles in the analysis of the TDS (\ref{tds}).

 		Notice that  $M(\lam)$ is a transcendental meromorphic matrix and hence has an essential singularity at infinity whereas the transfer function $G(\lam)$ of the PMD  $\mathbf{P}(\lam)$ is a rational matrix and has a pole at infinity. Therefore, important concepts such as the least order (i.e., total number of finite poles counting multiplicity)  and McMillan degree (i.e., total number of finite and infinite poles counting multiplicity) of $G(\lam)$, which play important roles in the analysis of the LTI system (\ref{lti}), are well defined for PMDs. By contrast, $M(\lam)$ has  either a finite number of poles in $\C$  and an isolated essential singularity at infinity or an infinite number of poles in $\C$ and a non-isolated essential singularity at infinity. 
 	
 	 For instance, the LTI system with a single control delay given by  
 	 \beano \frac{d\mathbf{x}}{dt} &=& A\mathbf{x}(t) +  B \mathbf{u}(t-\tau), \; t >0  \\ \mathbf{y}(t) &=& C \mathbf{x}(t) \eeano
 	 has been studied in~\cite{tdsbook2}. The system matrix and the transfer function are given by
 	 \be \label{cdelay} \H(z) := \left[ \begin{array}{c|c} z I_n -A  & Be^{-\tau z} \\ \hline -C & 0 \end{array} \right] \text{ and } M(z) = C ( z I_n -A)^{-1}B e^{-\tau z}.\ee In this case,  $M(z)$ is a transcendental meromorphic matrix with a finite number of poles in $\C$ and an isolated essential singularity at infinity.  On the other hand,  the  LTI system with a single state delay given by 
 	\beano \frac{d\mathbf{x}}{dt} &=& A\mathbf{x}(t) + A_d \mathbf{X}(t-\tau) + B \mathbf{u}(t), \; t >0  \\ \mathbf{y}(t) &=& C \mathbf{x}(t) \eeano
 	has been studied in~\cite{tdsbook1} using matrix Lambert W-function. The system matrix and the transfer function  are given by
 	\be \label{sdelay} \H(z) := \left[ \begin{array}{c|c} z I_n -A- A_d e^{-\tau z} & B \\ \hline -C & 0 \end{array} \right] \text{ and } M(z) = C ( z I_n -A- A_d e^{-\tau z})^{-1}B. \ee  In this case, $M(z)$ is a transcendental meromorphic matrix with an infinite number of poles in $\C$ and a non-isolated  essential singularity at infinity.

 	The TDS (\ref{tds}) is analyzed in \cite{dde1, dde2} in the special case when $ \tau_j = t_j = s_j =jh $ and $ D(\lam) = 0$. Introducing a new variable $ \mu := e^{-h \lam}$ so that $ \mu^j = e^{-jh \lam}$, the system matrix $\H(\lam)$ is considered as a matrix polynomial in two variables $\H(\lam, \mu)$, which introduces its own complexity,  and utilized Smith form of $\H(\lam, \mu)$ for analyzing the TDS.

  Motivated by the TDS (\ref{tds}), we undertake a comprehensive analysis of AMDs and matrix-fraction descriptions (MFDs) of meromorphic matrices. 	PMDs and MFDs of rational matrices  provide a powerful framework for analysis of LTI systems; see~\cite{rosenbrock70, kailath, vardulakis}. Our aim is to develop an analogous  framework for TDS and discuss  the extent to  which results of PMDs  can be generalized to  AMDs. 	We show that important results (e.g., coprime MFDs, least order  PMDs, equivalence of PMDs, canonical forms, characterization of PMDs by transfer functions, structural indices of zeros and poles)  which hold for PMDs can be generalized to AMDs. We mention that the generalizations are not routine and some results require special care. For instance, we show that the least order of $G(\lam)$ can indeed be generalized to $M(\lam)$ even when $M(\lam)$ has an infinite number poles and  show that the least order of $M(\lam)$ is directly related to the minimality  the AMD $\H(\lam).$ On the other hand, it is not clear how to generalize McMillan degree of $G(\lam)$ to $M(\lam)$ and what implication it has on the TDS (\ref{tds}). Further, the analysis of singular structures of $M(\lam)$ and $\H(\lam)$, namely, analogues of minimal bases and minimal indices, remains an open problem.

The rest of the paper is organized as follows. Section~2 is devoted to preliminaries which are needed in the rest of the paper. Section~3 analyzes matrix-fraction descriptions of meromorphic matrices and introduces least order of  a transcendental meromorphic matrix. Finally, Section~4 presents a detailed analysis of AMDs, their equivalence and characterization, canonical forms, zeros and poles including  their structural indices, and system of least order.

{\bf Notation:}  Let $\C[z]$ denote the ring of scalar polynomials with coefficients in $\C$  and  $\C(z)$ denote the field of rational functions of the form  $p(z)/ q(z),$ where $p(z)$ and $q(z)$ are  polynomials in $\mathbb{C}[z].$  We denote by $\C^{m\times n}, \mathbb{C}[z]^{m\times n}$ and  $\mathbb{C}(z)^{m\times n}$, the set of all $m\times n$ matrices with entries in $\C, \; \C[z]$ and $ \C(z)$, respectively. The elements of $\C[z]^{m\times n}$ are called matrix polynomials and the elements of $\C(z)^{m\times n}$ are called rational matrices. The $m\times n$ zero matrix is denoted by $0_{m\times n}$  as well as $0_{m, n}.$ The $n\times n$ identity matrix is denoted by $I_n$.

\section{Preliminaries}
Let $\mathcal{O} \subset \C$ be a {\em domain}, that is, $\mathcal{O}$ is a nonempty connected open set. Let $ X$ be a complex Banach space. A function $ f : \mathcal{O} \longrightarrow X$ is said to be holomorphic (or analytic) in $\mathcal{O}$ if $f$ is differentiable on $\mathcal{O}.$ We denote the set of all $X$-valued holomorphic (resp., meromorphic)  functions on $\mathcal{O}$ by $\mathbb{H}(\mathcal{O}, X)$ (resp., $\mathbb{M}(\mathcal{O})$), that is, 
\beano \mathbb{H}(\mathcal{O}, X) &:=& \{ f : \mathcal{O} \longrightarrow X\; \big \vert \; f \text{ is holomorphic in } \mathcal{O}\}, \\ 
\mathbb{M}(\mathcal{O}, X) &:=& \{ f : \mathcal{O} \longrightarrow X\; \big \vert \; f \text{ is meromorphic in } \mathcal{O}\}.
\eeano
Let $ \Omega \subset \C$ be a {\em region}, that is, the interior of $\Omega$ is a domain and $\Omega$ possibly contains a part or all of its boundary $\partial\Omega.$   A function $ f : \Omega \longrightarrow X$ is said to be holomorphic (resp., meromorphic)  in $\Omega$ if $ f \in \mathbb{H}(\mathcal{O}, X)$ (resp., $ f \in \mathbb{M}(\mathcal{O}, X)$) for some open set $ \mathcal{O} \subset \C$ such that $ \Omega \subset \mathcal{O}.$ We denote the set of all $X$-valued holomorphic (resp., meromorphic) functions on $\Omega$ by $\mathbb{H}(\Omega, X)$ (resp., $\mathbb{M}(\Omega, X)$).

For  simplicity of notation, when $ X = \C, X = \C^n$ and $ X = \C^{m\times n},$ we set  
\beano  \mathbb{H}(\Omega) &:=& \mathbb{H}(\Omega, \C),\;  \mathbb{H}(\Omega)^n :=  \mathbb{H}(\Omega, \C^n) \text{ and } \mathbb{H}(\Omega)^{m\times n} := \mathbb{H}(\Omega, \C^{m\times n}),\\
 \mathbb{M}(\Omega) &:=& \mathbb{M}(\Omega, \C),\;  \mathbb{M}(\Omega)^n :=  \mathbb{M}(\Omega, \C^n) \text{ and } \mathbb{M}(\Omega)^{m\times n} := \mathbb{M}(\Omega, \C^{m\times n}).\eeano 
We refer to the elements of $\mathbb{H}(\Omega)^n$ (resp., $\mathbb{M}(\Omega)^n$) as  holomorphic (resp., meromorphic) vectors  and the elements of $\mathbb{H}(\Omega)^{m\times n}$ (resp., $\mathbb{M}(\Omega)^{m\times n}$) as  holomorphic (resp., meromorphic)  matrices. Note that $ \mathbb{H}(\Omega) $ is a commutative ring with unity and $ \mathbb{H}(\Omega)^n $ is a module over the ring $ \mathbb{H}(\Omega).$ On the other hand, $ \mathbb{M}(\Omega) $  is the quotient field of the integral domain $\mathbb{H}(\Omega)$~(see,\cite{jali, rafiams, rmt}) and $ \mathbb{M}(\Omega)^n $ is a vector space over the field $ \mathbb{M}(\Omega).$

 The span of a subset $ \{ f_1, \ldots, f_{\ell}\} \subset \mathbb{H}(\Omega)^n$ over $\mathbb{H}(\Omega)$ and $\mathbb{M}(\Omega)$ are given by \beano \mathrm{span}_{\mathbb{H}(\Omega)}(f_1, \ldots, f_\ell) &:=& \{ g_1f_1 + \cdots+ g_\ell f_\ell : \;\;  g_1, \ldots, g_\ell \in  \mathbb{H}(\Omega)\}, \\ \mathrm{span}_{\mathbb{M}(\Omega)}(f_1, \ldots, f_\ell) &:=& \{ g_1f_1 + \cdots+ g_\ell f_\ell : \;\;  g_1, \ldots, g_\ell \in  \mathbb{M}(\Omega)\}.\eeano Obviously, $ \mathrm{span}_{\mathbb{H}(\Omega)}(f_1, \ldots, f_\ell)$ is a submodule of $\mathbb{H}(\Omega)^n$  and $ \mathrm{span}_{\mathbb{M}(\Omega)}(f_1, \ldots, f_\ell)$ is a subspace of $\mathbb{M}(\Omega)^n.$  A subset $ \{ f_1, \ldots, f_{\ell}\} \subset \mathbb{H}(\Omega)^n$ is said to be {\em linearly independent} in  $\mathbb{H}(\Omega)^n$ (resp.,  $\mathbb{M}(\Omega)^n$) if $g_1, \ldots, g_{\ell}$ in  $\mathbb{H}(\Omega)$ (resp., $\mathbb{M}(\Omega)$) and $ g_1f_1 + \cdots+ g_\ell f_\ell  =0$ then $ g_1 = \cdots = g_\ell = 0.$ It is easy to see~\cite{jali, rafiams} that $ \{ f_1, \ldots, f_{\ell}\} \subset \mathbb{H}(\Omega)^n$ is linearly independent in $\mathbb{H}(\Omega)^n$ if and only if  $ \{ f_1, \ldots, f_{\ell}\}$ is linearly independent in $\mathbb{M}(\Omega)^n$. Hence we have $$ \dim \mathrm{span}_{\mathbb{H}(\Omega)}(f_1, \ldots, f_\ell) = \dim \mathrm{span}_{\mathbb{M}(\Omega)}(f_1, \ldots, f_\ell).$$ Thus,  we can talk about linear independence of $ \{ f_1, \ldots, f_{\ell}\} \subset \mathbb{H}(\Omega)^n$ without mentioning either $\mathbb{H}(\Omega)^n$ or $\mathbb{M}(\Omega)^n.$ 
 
 \vone 
 
 \begin{definition} 
 	Let $ \mathcal{V} \subset \mathbb{M}(\Omega)^n$ be a subspace. Then  $ \mathcal{B} \subset \mathbb{H}(\Omega)^n$ is said to be an analytic basis of $\mathcal{V}$ if $\mathcal{B}$ is linearly independent and $\mathrm{span}_{\mathbb{M}(\Omega)}(\mathcal{B}) = \mathcal{V}.$  The dimension of $\mathcal{V}$ is the number of elements in $\mathcal{B}$ and is denoted by $\dim_{\mathbb{M}(\Omega)}(\mathcal{V}).$ 
 \end{definition} 
 
 \vone

 Let $  A \in \mathbb{H}(\Omega)^{m\times n}$ or  $  A \in \mathbb{M}(\Omega)^{m\times n}$. Then  $ A : \mathbb{M}(\Omega)^n \longrightarrow \mathbb{M}(\Omega)^m, f \longmapsto Af,$ is a linear transformation. The null space $N(A)$ and range space $R(A)$ are given by  $$ N(A) := \{ f \in \mathbb{M}(\Omega)^n : Af = 0\} \subset \mathbb{M}(\Omega)^n \text{ and } R(A) := \{ A f :  f \in \mathbb{M}(\Omega)^n\} \subset \mathbb{M}(\Omega)^m.$$   Then  $ \rank_{\mathbb{M}(\Omega)}(A) := \dim_{\mathbb{M}(\Omega)}( R(A))$ is called the {\em normal rank} of $A$ and is denoted by $ \nrank(A).$  Equivalently, $\nrank(A)$ is the number of linearly independent columns of  $A.$ Similarly, $\mathrm{nullity}(A) := \dim_{\mathbb{M}(\Omega)} (N(A))$ is called the {\em nullity} of $A.$ 
 The matrix $A$ is said to be {\em regular} if  $  \nrank(A) = m=n.$ On the other hand, $A$ said to be {\em singular} if $A$ is not regular.  Hence  $A$ is  singular when  $m\neq n$ or when $\nrank(A) < \min(m, n).$ In particular, if $ A \in \mathbb{H}(\Omega)^{m\times n}$ then  $\nrank(A) = \max\{ \rank(A(\omega)) : \; \omega \in \Omega\}.$

 \vone

 %Let $  A \in \mathbb{H}(\Omega)^{m\times n}$ or  $  A \in \mathbb{M}(\Omega)^{m\times n}$. 

\begin{definition} 
  Let $ A \in\mathbb{H}(\Omega)^{n\times n}.$ Then $A$ is said to be  invertible  if there exists $B \in \mathbb{H}(\Omega)^{n\times n}$ such that $ AB = BA = I_n,$ that is, $ A(z) B(z) = B(z)A(z) = I_n \; \text{ for all } \; z \in \Omega.$   In such a case, $B$ is called the inverse of $A$ and is denoted as $ A^{-1}.$  
\end{definition}

\von 
Note that if $A$ is invertible  then $ A^{-1}(z) = (A(z))^{-1}$ for all $z \in \Omega.$  Let  $\mathrm{GL}_n(\mathbb{H}(\Omega)) $ denote the group of invertible elements in $ \mathbb{H}(\Omega)^{n\times n}$, that is,
$$ \mathrm{GL}_n(\mathbb{H}(\Omega)) := \{ A \in \mathbb{H}(\Omega)^{n\times n} \; : \; A \text{ is invertible }\}.$$ 
The elements in $\mathrm{GL}_n(\mathbb{H}(\Omega)) \subset  \mathbb{H}(\Omega)^{n\times n}$ are called units or unit elements of $\mathbb{H}(\Omega)^{n\times n}.$ Note that the unit elements of $\C[z]^{n\times n}$ are unimodular matrix polynomials. A matix polynomial $P(z) \in \C[z]^{n\times n}$ is unimodular if $ \det(P(z)) $ is a nonzero constant~\cite{GLR}.

\vone

\begin{remark} \label{rem:rank} Let $ B :=\bmatrix{ f_1& \cdots & f_\ell} \in \mathbb{H}(\Omega)^{n \times \ell}.$ Then $\{f_1, \ldots, f_\ell\}$ is linearly independent $\Longleftrightarrow \nrank(B) = \ell.$ Hence $B$ is an analytic ordered basis of $R(B) \Longleftrightarrow \nrank(B) = \ell.$  Further, $B$ is said to be an invertible basis of $R(B)$ if $B$ is left invertible, that is, there exist $ B_L \in \mathbb{H}(\Omega)^{\ell \times n}$ such that $ B_L B = I_\ell$, that is, $ B_L(z) B(z) = I_\ell$ for all $ z \in \Omega.$
\end{remark} 

\von

\vone 
\begin{definition}[spectrum] Let $ A \in \mathbb{H}(\Omega)^{m\times n}$.   Then $\lam \in \Omega$ is said to be an eigenvalue of $A$ if $ \rank(A(\lam)) < \nrank(A).$  	The spectrum $\sig_{\Omega}(A)$ of $A$ is given by 
	$$\sig_{\Omega}(A) := \{ \lam \in \Omega : \rank(A(\lam)) < \nrank(A) \}. $$ The resolvent set $\rho_{\Omega}(A)$ is given by  $ \rho_{\Omega}(A) := \Omega \setminus \sig_{\Omega}(A).$
\end{definition}

\vone Note that if $A$ is regular then  $ \sig_{\Omega}(A) := \{ z \in \Omega : \det(A(z)) = 0 \}.$  Also, note that $ A \in  \mathrm{GL}_n(\mathbb{H}(\Omega))  \Longleftrightarrow \sig_{\Omega}(A) = \emptyset.$ For a nonzero scalar function $ f \in \mathbb{H}(\Omega) $ the spectrum $\sig_{\Omega}(f)$ is the set of zeros of $ f$, that is, $$\sig_{\Omega}(f) = \{ z \in \Omega : f(z) = 0\}.$$ Thus $\sig_{\Omega}(f) = \emptyset \Longleftrightarrow f$ is a unit element of $\mathbb{H}(\Omega) \Longleftrightarrow 1/f \in \mathbb{H}(\Omega).$ It is well known~\cite{alf} that $\sig_{\Omega}(f)$ is at most a countable set with no accumulation points in $\Omega.$ Consequently, $\sig_{\Omega}(A)$ is at most a countable set with no accumulation points in $\Omega$ when $ A \in \mathbb{H}(\Omega)^{n\times n}$ is regular.  In general, if $ A \in \mathbb{H}(\Omega)^{m\times n}$ and $ \nrank(A) = r$ then by considering common zeros of all $r\times r$ minors of $A$, it is follows $\sig_{\Omega}(A)$ is at most a countable set with no accumulation points in $\Omega.$

\vone 

Let $ M \in \mathbb{M}(\Omega)^{m\times n}.$ We denote the set of poles of $M$ by $\wp_{\Omega}(M)$, that is,
  $$ \wp_{\Omega}(M) := \{ \lam \in \Omega : \lam \text{ is a pole of } M(z)\}.$$  Obviously $ \wp_{\Omega}(M)$ is at most a countable set with no accumulation points in $\Omega.$ \\

\begin{definition}  Let $ M \in \mathbb{M}(\Omega)^{m\times n}.$ Then $ \lam \in \Omega$ is said to be an  eigenvalue of $M$ if $ \lam \notin \wp_{\Omega}(M)$ and $ \rank(M(\lam)) < \nrank(M).$   The  eigenspectrum $\eig_{\Omega}(M)$  of $M$ is given  by $$ \eig_{\Omega}(M) := \{ \lam \in \Omega : \;\; \lam \notin \wp_{\Omega}(M) \text{ and } \rank(M(\lam)) < \nrank(M)\}.$$

Let $ \mu \in \wp_{\Omega}(M).$ Then $ \mu$ said to be an   eigenpole of $M$ if there exists $ \mathbf{f} \in \mathbb{H}(\Omega)^n$ such that $$ \mathbf{f}(\mu) \neq 0, \mathbf{f} \notin N(M) \text{ and  } \lim_{z\rightarrow \mu} M(z)\mathbf{f}(z)  =0.$$

We denote the set of eigenpoles of $M$ by $\eip_{\Omega}(M).$ Then the spectrum $\sig_{\Omega}(M)$ of $M$ is given by  $ \sig_{\Omega}(M) := \eig_{\Omega}(M) \cup \eip_{\Omega}(M).$ If $ \lam \in \sig_{\Omega}(M)$ then $\lam$ is said to be a  zero of $M$ in $\Omega.$
\end{definition}

\vone Note that $\sig_{\Omega}(A)$ is at most a countable set with no accumulation points in $\Omega.$ We now consider analytic equivalence of holomorphic/meromorphic  matrices. See~\cite{ggk,gkl} for analytic equivalence of holomorphic operator-valued functions. \von

\begin{definition}\label{heqv}(a) Let $ U \subset \Omega$ be open and $ A, B \in \mathbb{M}(\Omega)^{m\times n}.$ Then $A$  and $B$ are said to be  equivalent  on $U$ and written as  $ A \sim_U B$ or $A(z) \sim_U B(z)$ if there exist  $E \in \mathrm{GL}_m(\mathbb{H}(U))$ and $ F \in \mathrm{GL}_n(\mathbb{H}(U))$ such that  $ A  = EBF$,     that is,  $A(z) =E(z) B(z) F(z)$   for all $ z \in U.$ 
	
(b)  Let $ \lam \in \Omega$  and $ A, B \in \mathbb{M}(\Omega)^{m\times n}.$ Then $A$  and $B$ are said to be equivalent  at $\lam$ and  written as	 $ A \sim_{\lam} B$ or $ A(z) \sim_{\lam} B(z)$  if there exists an open neighbourhood $U\subset \Omega$ of $\lam$ such that 
 $ A\sim_U B.$
\end{definition}

\vone 	
	
Observe that if  $ A \sim_{\Omega} B$    then  $\nrank(B) = \nrank(A).$  Hence   $A\sim_{\Omega} B \Longrightarrow \sig_{\Omega}(A) = \sig_{\Omega}(B)$ and $  \wp_{\Omega}(A) =  \wp_{\Omega}(B).$
	
\vone 

Let $  f \in \mathbb{H}(\Omega)$ and $ \lam \in \sig_{\Omega}(f).$ A positive integer $ \ell$  is said to be the order  of $\lam$ as a zero of $f$ if $ f^{(\ell)}(\lam) \neq 0$ and $ f^{(j)}(\lam) = 0 $ for $ j=1:\ell-1.$ Also, $\ell$ is called the multiplicity of $\lam$ as a zero of $f.$ On the other hand, if $ \mu \in \Omega$ and $ f(\mu) \neq 0$ then  $ \mu$ is said to be a zero of $ f$ of order $0.$ Thus $ \mu $ is a zero of $f$ of order $0 \Longleftrightarrow \mu \in \rho_{\Omega}(f).$ 
\von

We denote the set of integers by $\mathbb{Z}$ and the  set of non-negative integers by $ \mathbb{Z}_+$. 
Let  $ \nu : \Omega \longrightarrow \mathbb{Z}$. The support of $\nu$ denoted by $\mathrm{supp}(\nu) $  is given by $$ \mathrm{supp}(\nu) := \{ z \in \Omega: \nu(z) \neq 0 \}.$$  The support of $\nu$ is said to be {\em locally finite} if $ \mathrm{supp}(\nu) \cap \mathrm{K}$ is a finite set for every compact set $ \mathrm{K} \subset \Omega$~\cite{rmt}.

\begin{definition}[divisor,\cite{rmt}] A function   $ \nu : \Omega \longrightarrow \mathbb{Z}$  is called a divisor on $\Omega$ if $\mathrm{supp}(\nu)$   is locally finite.  A divisor $\nu$ is called a  principal divisor  if there exists $f \in \mathbb{M}(\Omega)$ such that 
$$ \nu(z) = \left\{ \begin{array}{rl} p & \text{ if } p \text{ is the order of } z \text{ as a zero of } f,\\ -p &   \text{ if } p \text{ is the order of } z \text{ as a pole of } f. \end{array}   \right.$$

\end{definition} 
\vone 

The set of all divisors on $\Omega$ is an additive abelian group~\cite{jali,rafiams, rmt}.
Let $f , g \in \mathbb{H}(\Omega)$. Then $f$ is called a divisor of $g$ if $g/f \in \mathbb{H}(\Omega),$ that is, there exist $h \in \mathbb{H}(\Omega)$ such that $  g = f h.$ A nonzero non-unit $u \in \mathbb{H}(\Omega)$ is called a prime of $\mathbb{H}(\Omega)$ if $u$ divides $fg$ for some $ f, g \in \mathbb{H}(\Omega)$ then $u$ divides $f$ or $u$ divides $g.$ The functions $\phi_\lam: z \mapsto (z- \lam)$ for  $ \lam \in \Omega,$ up to unit factors, are the primes of $\mathbb{H}(\Omega).$  

\vone 
\begin{remark} 
If $ \nu$ is the principal divisor of $f \in \mathbb{M}(\Omega)$ then $f $ is holomorphic $\Longleftrightarrow \nu \geq 0.$  Further, if $ f_1, f_2 \in \mathbb{H}(\Omega) $ with principal divisors $\nu_1$ and $\nu_2,$ then $ f_2/f_1 \in \mathbb{H}(\Omega) \Longleftrightarrow \nu_1 \leq \nu_2.$   Note that $\nu_{\min} := \min(\nu_1, \nu_2)$ and $\nu_{\max} := \max( \nu_1, \nu_2)$ are divisors on $\Omega.$  If $f \in \mathbb{H}(\Omega)$ with the principal divisor  $\nu_{\min} := \min(\nu_1, \nu_2)$ then $ f = \mathrm{gcd}(f_1, f_2).$  Hence $f_1$ and $f_2$ are coprime $\Longleftrightarrow  \nu_{\min} =0 \Longleftrightarrow \mathrm{gcd}(f_1, f_2)$ is a unit in  $\mathbb{H}(\Omega).$ Similarly, if $g \in \mathbb{H}(\Omega)$ with the principal divisor  $\nu_{\max} := \min(\nu_1, \nu_2)$ then $ g = \mathrm{lcm}(f_1, f_2).$ Further, $f, g \in \mathbb{M}(\Omega)$  have the same principal divisor $\nu$ on $\Omega \Longleftrightarrow g = uf $ for some unit element $u$ of $ \mathbb{H}(\Omega).$ 
See~\cite{jali, rafiams, rmt} for details.
\end{remark} 
\vone 

%
%
%If $ \partial$ is the principal divisor of $f$ on $\Omega$ with finite support and  $\mathrm{supp}(\partial) =\{ \lam_1, \ldots, \lam_\ell\}$ then $ f(z) = \prod^\ell_{j=1} (z- \lam_j)^{\partial(\lam_j)} g(z)$ for some unit element $ g \in \mathbb{H}(\Omega).$ In contrast, if $ \mathrm{supp}(\partial) = \{ \lam_n : n \in \mathbb{N}\} $ then an infinite product representation of $f$ of the form $ f(z) = \prod^\infty_{n=1}(z-\lam_n)^{\partial(\lam_n)} g(z)$ for some unit element $ g \in \mathbb{H}(\Omega)$ may not exists. However, an infinite product representation of $f$ can be obtained by replacing the factors $(z- \lam_n)^{\partial(\lam_n)}, n \in \mathbb{N},$ with so called Weierstrass factors, which are defined as follows. \\ 

\begin{definition} \cite{rmt}
An infinite  product $\prod^\infty_{n=1} f_n$ with $ f_n = 1+g_n \in \mathbb{H}(\Omega)$ is said to be  normally convergent in $\Omega$ if the series $\sum^\infty_{n=1} g_n$ converges absolutely on every compact set $\mathrm{K} \subset \Omega.$ \\
\end{definition}

\begin{theorem}[Smith form, \cite{rafiams}]\label{hsmith} Let $ A \in \mathbb{H}(\Omega)^{m\times n}.$ Suppose that $nrank(A) = r$.   Then there exist holomorphic functions $ \phi_1, \ldots, \phi_r$ in $\mathbb{H}(\Omega) $ with principal divisors $\nu_1, \ldots, \nu_r $ on $\Omega$ such that $\nu_1 \leq \cdots \leq \nu_r$ and 
	$$ A(z) \sim_{\Omega} \left[\begin{array}{ccc|c}\phi_1(z) & & & \\ & \ddots& &  \\ & & \phi_r(z) & \\ \hline & & & \mathbf{0}_{m-r, n-r} \end{array}\right] =: S_A(z).$$ The functions $\phi_1, \ldots, \phi_r$ are unique up to unit elements of $ \mathbb{H}(\Omega)$ and $ \phi_j$ divides $\phi_{j+1}$ for $ j=1:r-1.$ Further, let $ \sig_{\Omega}(A) = \{ z_n : n \in \mathbb{N}\}.$ Then 	$$ \phi_j(z) =  \prod^\infty_{\ell=1} (z- z_\ell)^{\nu_j(z_\ell)}u_{j\ell}(z) \;\text{ for all } z \in \Omega \text{ and } j=1:r,$$ where $ u_{j\ell} \in \mathbb{H}(\Omega)$ are unit elements, $\ell \in \mathbb{N}$    and $\prod^\infty_{\ell=1} (z- z_\ell)^{\nu_j(z_\ell)}u_{j\ell}(z)$  (possibly empty product)  converges normally in $\Omega$.  
\end{theorem} 

\vone

The diagonal matrix  $S_A(z)$ is the {\em  Smith  form} of $A(z)$ on $\Omega.$ The functions $\phi_1, \ldots, \phi_r$ are the {\em invariant functions}  of $A(z)$ on $\Omega.$   If $\phi_j$ is not a unit element then it is called a non-unit invariant function or a non-unit invariant factor of $A(z)$. Define $ \phi_A(z) := \prod^r_{j=1}\phi_j(z)$ for $ z \in \Omega.$ Then $\phi_A$  is unique up to a unit element of $\mathbb{H}(\Omega)$ and is called the {\em zero function} of $A(z)$  on $\Omega.$  It follows that $ \lam \in \Omega$ is an eigenvalue value of $A(z) \Longleftrightarrow \phi(\lam) =0.$ Hence $\sig_{\Omega}(A) = \{ \lam \in \Omega : \phi_A(\lam) = 0\}.$  The tuple $(\nu_1, \ldots, \nu_r)$ is the {\em structural invariant of zeros} of $A(z)$ which yields the structural indices of zeros of $A(z).$

{\bf Zero index:} The tuple $ \mathrm{Ind}_e(\lam, A) := ( \nu_1(\lam), \ldots, \nu_r(\lam)) \in \mathbb{Z}_+^r$ for $ \lam \in \sig_{\Omega}(A)$ is called the {\em zero index} of $A$ at $ \lam.$ 
 Observe that there exists $0 \leq \ell < r $ such that $$  0 = \nu_1(\lam) = \cdots = \nu_\ell(\lam) < \underbrace{\nu_{\ell+1}(\lam) \leq \cdots \leq \nu_r(\lam)}_{\text{ partial multiplicities}} \text{ and }  m_e(\lam) := \underbrace{\nu_1(\lam)+\cdots +\nu_r(\lam)}_{\text{total multiplicity}} $$ and $\nu_r(\lam)$ is the  ascent of $\lam.$ Thus $(z- \lam)^{\nu_{\ell+1}(\lam)}, \ldots, (z- \lam)^{\nu_{r}(\lam)}$ are elementary divisors of $A(z)$ at $ \lam.$ Moreover, $A(z) \sim_{\lam} \left[ (z-\lam)^{\nu_1(\lam)} \oplus \cdots \oplus (z- \lam)^{\nu_r(\lam)} \oplus 0_{m-r, n-r}\right].$
\vone 

%We mention that a global canonical form of $\M\in\mathbb{M}(\Omega)^{m\times n}$ has been studied~\cite{leiter} in which  $\M(z)$ is shown to be equivalent to a diagonal  matrix $ \diag(f_1(z), \ldots, f_r(z), 0, \ldots, 0)$, where $ f_1(z), \ldots,  f_r(z)$ are meromorphic functions in $\mathbb{M}(\Omega)$ such that $ f_i(z)$ divides $f_{i+1}(z)$ for $ i=1:r-1.$ Also,  compare our Theorem~\ref{hsmith} with \cite[Theorem~2.7]{tiss1}.  \vone 

%We mention that the global Smith form of a regular holomorphic matrix-valued function is considered~\cite{tiss1}. Compare our Theorem~\ref{hsmith} with \cite[Theorem~2.7]{tiss1}.  

\vone

Since  $ A(z) \sim_\Omega S_A(z)$,  there exist $ E \in \mathrm{GL}_m(\mathbb{H}(\Omega))$ and $ F \in \mathrm{GL}_n(\mathbb{H}(\Omega))$ such that $$  A(z) F(z) = E(z)S_A(z) \;\text{  for all } \;  z \in \Omega.$$ Hence the last $n-r$ columns of $F(z)$ form a basis of $N(A)$ and the first $r$ columns of $E(z)$ form a basis of $R(A).$ Indeed, let $$ F(z) = \bmatrix{ v_1(z) & \cdots & v_n(z)}  \text{ and } E(z) = \bmatrix{ u_1(z) & \cdots & u_m(z)}$$ be column  partitions of $F(z)$ and $ E(z)$. Then $ Av_j = 0 $ for $ j=r+1: n$ which shows that $N(A) = \mathrm{span}_{\mathbb{M}(\Omega)}( v_{r+1}, \ldots, v_n)$ and  $\{ v_{r+1}, \ldots, v_n\}$ is an invertible  basis of $N(A).$ Next, note that $ Av_j = \phi_j u_j$ for $ j=1: r$ which  shows that $R(A) = \mathrm{span}_{\mathbb{M}(\Omega)}( u_{1}, \ldots, u_r)$ and $\{u_1, \ldots, u_r\} $ is an  invertible basis of $R(A).$  

\vone

\begin{theorem}[Smith-McMillan form, \cite{rafiams}]\label{smc} 
	Let $ A \in \mathbb{M}(\Omega)^{m\times n}$. Suppose that $\nrank(A) = r.$ Then the following hold: 
	\begin{itemize} 	
		\item[(a)] There exist  functions $ \phi_1, \ldots, \phi_r$ in $\mathbb{H}(\Omega) $ with principal divisors $\nu_1, \ldots, \nu_r $ on $\Omega$ such that $\nu_1 \leq \cdots \leq \nu_r.$  
		
		\item[(b)] There exist  functions $ \psi_1, \ldots, \psi_r$ in $\mathbb{H}(\Omega) $ with principal divisors $\kappa_1, \ldots, \kappa_r $ on $\Omega$ such that $\kappa_1 \geq \cdots \geq \kappa_r.$

		\item[(c)] $\phi_j$ and $\psi_j$ are relatively prime, that is, $\mathrm{gcd}(\phi_j, \psi_j) =1$ for $ j=1:r$ and 
		$$ A(z) \sim_{\Omega} \left[\begin{array}{ccc|c}\phi_1(z)/\psi_1(z) & & & \\ & \ddots& &  \\ & & \phi_r(z)/\psi_r(z) & \\ \hline & & & 0_{m-r, n-r} \end{array}\right] =: \Sigma_A(z),$$ where $\phi_1, \ldots, \phi_r$ and $ \psi_1, \ldots, \psi_r$ are unique up to unit elements of $ \mathbb{H}(\Omega).$  % Further, $ \phi_j$ divides $\phi_{j+1}$  and $ \psi_{j+1}$ divides $\psi_j$ for $ j=1:r-1.$
		
		\item[(d)]  If $ \sig_{\Omega}(A) = \{ \lam_n : n \in \mathbb{N}\} \text{ and } \wp_{\Omega}(A) = \{ \mu_n : n \in  \mathbb{N}\},$ then
		 $$ \phi_j(z) =  \prod^\infty_{\ell=1} (z- \lam_\ell)^{\nu_j(\lam_\ell)}u_{j\ell}(z) \text{ and  }  \psi_j(z) =  \prod^\infty_{\ell=1} (z- \mu_\ell)^{\kappa_j(\mu_\ell)}v_{j\ell}(z) \;\text{ for } j=1:r,$$where $  u_{j\ell}, v_{j\ell} \in \mathbb{H}(\Omega)$ are unit elements,  $\ell \in \mathbb{N}$ and the infinite products (possibly empty) converge normally in $\Omega.$ 
	\end{itemize}
	
\end{theorem}

\vone 
The diagonal matrix $\Sigma_A(z)$ is the {\em Smith-McMillan form} of $A(z)$ on $\Omega.$ The functions $\phi_1, \ldots, \phi_r$ are called the {\em  invariant zero functions}  of $A(z)$  and  $\psi_1, \ldots, \psi_r$ are called the {\em invariant pole functions}  of $A(z)$ on $\Omega$.   Define $ \phi_A(z) := \prod^r_{j=1} \phi_j(z)$ and $ \psi_A(z) := \prod^r_{j=1}\psi_j(z).$ Then it follows that  $$  \wp_{\Omega}(A) =\{ \mu \in \Omega : \psi_A(\mu) = 0\} = \sig_{\Omega}(\psi_A) \text{ and } \sig_{\Omega}(A) = \{ \lam \in \Omega : \phi_A(\lam) = 0\} = \sig_{\Omega}(\phi_A).$$  In fact, we have~(see, \cite{jali, rafiams}) $\sigma_{\Omega}(A )=  \sig_{\Omega}(\phi_A) = \eig_{\Omega}(A) \cup \eip_{\Omega}(A)$ and 	\begin{align*}
	\eig_{\Omega}(A) & = \{ \lambda \in \Omega \ : \ \phi_A(\lambda)=0 \ \emph{and} \ \psi_{A}(\lambda) \neq 0 \}, \\  \eip_{\Omega}(A)&= \{ \lambda \in \Omega  : \ \phi_A(\lambda)=0 \ \emph{and} \ \psi_{A}(\lambda) = 0 \}.
\end{align*}

The tuple $(\nu_1, \ldots, \nu_r)$ is the {\em structural invariant of zeros} of $A(z)$ and the tuple  $ (\kappa_1, \ldots, \kappa_r)$ is the {\em structural invariant of poles} of $A(z)$ which yield the structural indices of zeros and poles of $A(z).$   

\von
{\bf  Zero index:} The tuple $ \mathrm{Ind}_e(\lam, A) := ( \nu_1(\lam), \ldots, \nu_r(\lam)) \in \mathbb{Z}_+^r$ for $ \lam \in \sig_{\Omega}(A)$ is the {\em zero index} of $A$ at $\lam.$  Then there exists $0 \leq \ell < r $ such that $$  0 = \nu_1(\lam) = \cdots = \nu_\ell(\lam) < \underbrace{\nu_{\ell+1}(\lam) \leq \cdots \leq \nu_r(\lam)}_{\text{ partial multiplicities}} \text{ and }  m_e(\lam) := \underbrace{\nu_1(\lam)+\cdots +\nu_r(\lam)}_{\text{total multiplicity}}$$ and $\nu_r(\lam)$ is the ascent of $\lam.$ 
\vone

{\bf Pole index:} The tuple $ \mathrm{Ind}_p(\lam, A) := ( \kappa_1(\lam), \ldots, \kappa_r(\lam)) \in \mathbb{Z}_+^r$ for $ \lam \in \wp_{\Omega}(A)$ is the {\em pole index} of $A$ at $\lam.$  Observe that there exists $0 \leq \ell \leq r $ such that $$ \underbrace{ \kappa_1(\lam) \geq  \cdots \geq \kappa_\ell(\lam)}_{\text{ partial multiplicities}}  > \kappa_{\ell+1}(\lam) = \cdots = \kappa_r(\lam) =0\text{ and }  m_p(\lam) := \underbrace{\kappa_1(\lam)+\cdots +\kappa_r(\lam)}_{\text{total multiplicity }}$$ and $\kappa_1(\lam)$ is the order of the pole $\lam.$  \vone 

{\bf Pole-Zero index:} $ \mathrm{Ind}_{ep}(\lam, A) := ( \tau_1(\lam), \ldots, \tau_r(\lam)) \in \mathbb{Z}^r$ for $ \lam \in \sig_{\Omega}(A)\cup \wp_{\Omega}(A)$ is the {\em pole-zero index} of $A$ at $ \lam,$ where  $ \tau_j(\lam) := \nu_j(\lam) - \kappa_j(\lam)$ for $ j=1:r.$  If $\lam \in \eip_{\Omega}(A) $ then there exist positive integers $ 1\leq  \ell <  s \leq r$ such that $$ \underbrace{\tau_1(\lam) \leq \cdots \leq \tau_\ell(\lam)}_{\text{pole partial multiplicities}} <  0 =\cdots = 0 <  \underbrace{ \tau_s(\lam) \leq \cdots \leq \tau_r(\lam)}_{\text{zero partial multiplicities}}. $$

For $\lam \in \C$, we have $ A(z) \sim_\lam \left[ (z-\lam)^{\tau_1(\lam)}\oplus \cdots\oplus (z-\lam)^{\tau_r(\lam)}\oplus 0_{m-r, n-r}\right]$ which is the local Smith form of $A$. See~\cite[p.414]{kazlov} and \cite{gks} for local Smith form of a holomorphic matrix.

\vone

\section{Matrix-fraction description (MFD)} MFDs of rational matrices are powerful tools for analyzing PMDs. In this section, we develop analogous theory for MFDs of meromorphic matrices.

Let $ A \in \mathbb{H}(\Omega)^{m\times n}$ and $ D \in \mathbb{H}(\Omega)^{n\times n}.$ Then $ D$ is said to be a {\em  right divisor} of $A$ if there exists $ Q \in \mathbb{H}(\Omega)^{m\times n}$ such that $ A = QD.$ Let $  B \in \mathbb{H}(\Omega)^{p\times n}.$ If $D$ is also a right divisor of $B \in \mathbb{H}(\Omega)^{p\times n}$ then $ D$ is said to be a { \em common right divisor} of $A$ and $B.$ Further, $D$ is said to be a { \em greatest common divisor (gcrd)} of $A$ and $B$ if $N\in \mathbb{H}(\Omega)^{n\times n} $ is a common right divisor of $A$ and $B$ then $N$ is a right divisor of $D.$ Furthermore, $A$ and $B$ are said to be { \em right coprime} if every  gcrd of $A$ and $B$ are invertible, that is, if $D$ is a gcrd of $A$ and $ B$ then $ D \in \mathrm{GL}_n(\mathbb{H}(\Omega)).$  Left divisors, common left divisors, greatest common left divisors, and left coprime matrices are defined similarly.

\vone  We mention that  a common right divisor of $A$ and $B$ can be derived from the Smith form. Indeed, let  $ E(z) A(z) F(z) =   S_A(z)$ be the Smith form of $A(z).$  Then setting $ D(z) := F(z)^{-1}$ and  $Q(z) := E(z)^{-1}S_A(z)$,  we have  $ A(z) =  Q(z)D(z)$ which  shows that $D$ is a right divisor of $A.$ Next, let $S(z)$ be the Smith form of $ \mathcal{W} := \bmatrix{ A(z) \\ B(z)}.$ Then $ \mathcal{W} = E(z) S(z) F(z),$ for some  $E \in \mathrm{GL}_{m+p}(\mathbb{H}(\Omega))$ and $F \in \mathrm{GL}_{n}(\mathbb{H}(\Omega)).$   Let $Q_1(z) $ (resp.,  $Q_2(z)$)  denote the first $m$ (resp., last $p$) rows of $ E(z)S(z).$ Set $D(z) := F(z)$ Then we have $$ A(z) = Q_1(z) D  \; \text{ and } \; B(z) = Q_2(z)D(z)$$ which show that $D(z)$ is a common right divisor of $A(z)$ and $B(z).$  A gcrd of two holomorphic matrices can also be extracted from the Smith form.

%In fact, $D(z)$ is a gcrd of $A(z)$ and $B(z).$ 

\vone 

Let  $S_A(z)$ be the Smith form of $A.$ Then  $ A(z) = E(z) S_A(z) F(z)$ for some $ E \in \mathrm{GL}_m(\mathbb{H}(\Omega)) $ and $ E \in \mathrm{GL}_n(\mathbb{H}(\Omega))$. We refer to the decomposition $ A(z) = E(z) S_A(z) F(z)$ as the Smith decomposition of $A(z).$ Now if $ \nrank(A) = r$ then $$ S_A(z) = \bmatrix{D(z) & 0  \\ 0 & 0} = \bmatrix{ I_r \\ 0 } \bmatrix{ D(z) & 0 },$$ where $ D(z) := \diag(\phi_1(z), \ldots, \phi_r(z))$ and $\phi_1, \ldots, \phi_r$ are invariant functions of $A.$  Consider the conformal row and column partitions $  E = \bmatrix{ E_1 & E_2}$ and $ F= \bmatrix{ F_1 \\ F_2}.$ Then we have $ A(z) = E_1(z) D(z) F_1(z), $ where $ E_1$ is left invertible and $F_1$ is right invertible, that is, there exist $ E_L \in \mathbb{H}(\Omega)^{r\times m}$ and $F_R \in \mathbb{H}(\Omega)^{r\times n}$ such that $ E_L(z) E_1(z) = I_r$ and $ F_1(z) F_R(z) = I_r$ for all $ z \in \Omega.$ Define $ D_R(z) := D(z) F_1(z). $ Then we have a (non-unique) decomposition \be\label{rd}  A(z) = E(z) \bmatrix{ D_R(z) \\ 0 } = E_1(z) D_R(z),\ee where $E_1(z)$ is left invertible. 
We refer to $D_R(z)$ as a {\bf \em right-structure matrix} of $A(z).$ Observe that the Smith forms of  $A(z)$ and $D_R(z)$ have the same non-unit invariant functions. Consequently, we have $ \sig_{\Omega}(A) = \sig_{\Omega}(D_R)$ and $ \mathrm{Ind}_e(\lam, A) = \mathrm{Ind}_e(\lam, D_R)$ for all $ \lam \in \sig_{\Omega}(A).$ Thus $A(z)$ and $D_R(z)$ have the same spectral structure, that is, $A(z)$ and $D_R(z)$ have the same  eigenvalues including their partial multiplicities.  This proves the following result.

\vone 

\begin{proposition} \label{rstruc} Let  $ A \in \mathbb{H}(\Omega)^{m\times n}$ be such that $ \nrank(A) =r.$  Then $A$ can be decomposed as $ A(z) = E(z) D_R(z)$, where $ E \in \mathbb{H}(\Omega)^{m\times r}$ is left invertible and $D_R \in  \mathbb{H}(\Omega)^{r\times n}$ is a right-structure matrix of $A(z).$ Further, the Smith forms of $A(z)$ and $D_R(z)$ have the same non-unit invariant functions. In particular, 
  $ \sig_{\Omega}(A) = \sig_{\Omega}(D_R)$ and $ \mathrm{Ind}_e(\lam, A) = \mathrm{Ind}_e(\lam, D_R) \; \text{  for all } \; \lam \in \sig_{\Omega}(A).$
\end{proposition}

\vone The next result extracts a gcrd of two holomorphic matrices. \vone 

\begin{theorem}\label{gcrd}  Let $ A \in \mathbb{H}(\Omega)^{m\times n}$ and $  B \in \mathbb{H}(\Omega)^{p\times n}$. Set $\mathcal{W}(z) := \bmatrix{ A(z) \\ B(z)}.$  Suppose that $\nrank(\mathcal{W}) = n.$ 	
	 Let $ D_R \in \mathbb{H}(\Omega)^{n\times n}$ be a right-structure matrix of $\mathcal{W}(z)$. Then $ D_R(z)$ is a gcrd of $A(z)$ and $B(z).$ 
\end{theorem}

\von 
\begin{proof} Note that there is a matrix $ E \in \mathrm{GL}_{m+p} (\mathbb{H}(\Omega))$ such that $ \mathcal{W}(z) = E(z) \bmatrix{ D_R(z) \\ 0}.$ Indeed, consider the Smith decomposition  $ \mathcal{W}(z) = E(z) \bmatrix{ D(z) \\ 0 } F(z)  $, where $ D \in \mathbb{H}(\Omega)^{n\times n}$ is given by $ D(z) := \diag(\phi_1(z), \ldots, \phi_n(z)). $   Then setting $ D_R(z) := D(z) F(z),$ we have the desired result.

Now consider the conformal partition $ E = \bmatrix{ E_1 & G_1\\ E_2 & G_2}$, where $ E_1 \in \mathbb{H}(\Omega)^{m\times n}$ and $ E_2 \in \mathbb{H}(\Omega)^{p\times n}.$ 	Then $ \bmatrix{ E_1\\ E_2}$ is left invertible and $ \mathcal{W}(z) = \bmatrix{ E_1(z) \\ E_2(z)} D_R(z)$ which shows that $ A(z) = E_1(z) D_R(z)$ and $ B(z) = E_2(z) D_R(z).$ Hence $ D_R(z)$ is a common right divisor of $A(z)$ and $B(z).$

Next, consider the conformal partition $ E^{-1} = \bmatrix{  X_1 & X_2 \\ Y_1 & Y_2}$. Then $E^{-1}\mathcal{W} = \bmatrix{D_R \\ 0 }$ yields $ X_1(z) A(z) + X_2(z) B(z) = D_R(z).$ If $ N(z)$ is a common right divisor of $A(z)$ and $B(z)$ then $ A(z) = Q_1(z) N(z)$ and $ B(z) = Q_2(z) N(z).$  Consequently, we have $ D_R(z) = X_1(z) A(z) + X_2(z) B(z) = (X_1(z)Q_1(z) + X_2(z) Q_2(z)) N(z)$ showing that $ N(z)$ is a right divisor of $D_R(z).$ This proves that $D_R(z)$ is a gcrd of $A(z)$ and $B(z).$ 	
\end{proof}

\vone 
Recall that $A(z)$ and $B(z)$ are said to be right coprime if every gcrd of $A(z)$ and $B(z)$ is invertible. The next result characterizes right coprime matrices. \von

\begin{proposition}\label{coprime} Let $ A \in \mathbb{H}(\Omega)^{m\times n}$ and $ B \in \mathbb{H}(\Omega)^{p\times n}$. Then 
	$A(z)$ and $B(z)$ are right coprime $\Longleftrightarrow \rank \bmatrix{A(z)\\ B(z) } = n$ for all $z \in \Omega.$  
\end{proposition}
\von 

 \begin{proof} Suppose that $ A(z)$ and $B(z)$ are right coprime. Let $ D$ be a gcrd of $A(z)$ and $B(z)$. Then $D \in \mathrm{GL}_n(\mathbb{H}(\Omega)).$  Further, there exist $Q_1 \in \mathbb{H}(\Omega)^{m\times n}$ and $Q_2 \in \mathbb{H}(\Omega)^{p\times n}$ such that $\bmatrix{Q_1\\ Q_2}$ is left invertible and $\bmatrix{A\\ B} = \bmatrix{Q_1\\ Q_2} D.$ Now $ \rank\bmatrix{ A(z) \\ B(z) } = \rank\bmatrix{Q_1(z) \\ Q_2(z)} = n$ for all $ z \in \Omega.$

Conversely, suppose that $\rank \bmatrix{ A(z) \\ B(z) } = n$ for all $ z \in \Omega.$  Set $ \mathcal{W} := \bmatrix{ A\\ B}.$ Then $\rank \,\mathcal{W}(z) = n$ for all $ z \in \Omega \Longrightarrow \nrank (\mathcal{W}) = n.$  Let $D$ be a right-structure matrix of $\mathcal{W}(z).$ Then by Theorem~\ref{gcrd}, $D(z)$ is a gcrd of $ A(z)$ and $B(z).$  By Proposition~\ref{rstruc}, we have $  \sig_{\Omega}(D) = \sig_{\Omega}(\mathcal{W}) = \emptyset$ which shows that $ D$ is invertible, that is, $ D \in \mathrm{GL}_n(\mathbb{H}(\Omega)).$ \end{proof} 

\von

Further, we have following result which will play an important role in the subsequent development.

\begin{theorem}\label{cprime2} Let $ A \in \mathbb{H}(\Omega)^{m\times n}$ and $ B \in \mathbb{H}(\Omega)^{p\times n}.$ Set $ \mathcal{W} := \bmatrix{ A\\ B}.$ Then the following statements are equivalent. 
	\begin{itemize}
		\item[(a)] $A(z)$ and $B(z)$ are right coprime \item[(b)] $\sig_{\Omega}(\mathcal{W}) = \emptyset. $ \item[(c)] There exists $ E \in \mathrm{GL}_{m+p}(\mathbb{H}(\Omega))$ such that $ E(z) \mathcal{W}(z) = \bmatrix{ I_n\\ 0} = S_{\mathcal{W}}(z),$ where $S_{\mathcal{W}}(z)$ is the Smith form of $\mathcal{W}(z).$ 
		
	\item[(d)] There exist $ X \in \mathbb{H}(\Omega)^{n\times m}$ and $ Y \in \mathbb{H}(\Omega)^{n\times p}$ such that $ X(z) A(z) + Y(z) B(z) = I_n$ for all $ z \in \Omega.$  
	\item[(e)] There exist holomorphic matrices $C(z)$ and $ D(z)$ such that $ \bmatrix{ A & C \\ B & D} \in \mathrm{GL}_{m+p}(\mathbb{H}(\Omega)).$ 
	\end{itemize}
\end{theorem}

\vone \begin{proof} The proof follows from Smith form of $\mathcal{W}(z)$ and similar arguments as those used in the proofs of Theorem~\ref{gcrd} and Proposition~\ref{coprime}. \end{proof} 
\vone

We now explore solution of a matrix equation of the form $$ X(z) A(z) + Y(z) B(z)  = C(z)   \text{ or } \bmatrix{X(z) & Y(z)} \bmatrix{A(z)\\ B(z)} = C(z),$$ where the matrices $ A \in \mathbb{H}(\Omega)^{m\times n}, B \in \mathbb{H}(\Omega)^{p\times n}$ and $ C \in \mathbb{H}(\Omega)^{r\times n}$ are given and the matrices $ X \in \mathbb{H}(\Omega)^{r\times m}$ and $ Y \in \mathbb{H}(\Omega)^{r\times n}$ are  sought to be determined. 

\vone 

\begin{theorem}\label{bez} Let $ A \in \mathbb{H}(\Omega)^{m\times n}, B \in \mathbb{H}(\Omega)^{p\times n}$ and $ C \in \mathbb{H}(\Omega)^{r\times n}$. 	Then the  matrix equation  $X(z) A(z) + Y(z) B(z)  = \bmatrix{X(z) & Y(z)} \bmatrix{A(z)\\ B(z)} = C(z)$ has a solution $ \bmatrix{X(z) & Y(z)} \Longleftrightarrow$ every gcrd of $A(z)$ and $B(z)$ is a right divisor of $C(z).$

In particular, if $A(z)$ and $B(z)$ are right coprime, then the matrix equation always has a solution. 
\end{theorem}

\von \begin{proof}  Suppose that a solution $\bmatrix{X(z) & Y(z)}$ exists. If $D(z)$ is a gcrd of $ A(z)$ and $B(z)$ then it follows that $ D(z)$ is a right divisor of $C(z).$

Conversely, suppose that every gcrd of $A(z)$ and $B(z)$ is a right divisor of $ C(z).$ Let $ D(z)$ be a gcrd of $A(z)$ and $ B(z)$. Then $ C(z) = R(z) D(z).$ Set $ \mathcal{W}(z) := \bmatrix{ A(z) \\ B(z)}.$ Then by (\ref{rd}), we have  $$ E(z) \mathcal{W}(z) = \bmatrix{ D(z) \\ 0},$$ where $ E$ is invertible. Consider the conformal partition $E(z) = \bmatrix{ Q_1(z) & Q_2(z) \\ P_1(z) & P_2(z)}.$ Then $ Q_1(z) A(z) + Q_2(z) B(z) = D(z).$  Left multiplying on both side by $R(z)$, we have   $$  R(z) Q_1(z) A(z) + R(z) Q_2(z) B(z) = R(z)D(z) = C(z).$$ Setting $ X(z) := R(z) Q_1(z)$ and $ Y(z) := R(z)Q_2(z),$ we have $$ X(z)A(z) + Y(z)B(z) = C(z).$$ Hence $ \bmatrix{X(z) & Y(z)}$ is a solution of the matrix equation. 

\vone
Now, if $A(z)$ and $B(z)$ are right coprime then a gcrd $D(z)$ is invertible. Then $ Q_1(z) A(z) + Q_2(z) B(z) = D(z) \Longrightarrow D(z)^{-1}Q_1(z) A(z) + D(z)^{-1}Q_2(z) B(z) = I_n.$ Left multiplying on both sides by $C(z)$, we have $$ C(z)D(z)^{-1}Q_1(z)A(z) + C(z) D(z)^{-1}Q_2(z) B(z) = C(z).$$ Setting $ X(z) := C(z)D(z)^{-1}Q_1(z) $ and $ Y(z) := C(z) D(z)^{-1}Q_2(z)$, we obtain a solution $\bmatrix{X(z) & Y(z)}$ of the  matrix equation. 
\end{proof}

\vone

If $ f \in \mathbb{M}(\Omega)$ then there exist $ p, q \in \mathbb{H}(\Omega)$ such that $ p$ and $q$ are coprime and  $\displaystyle f = \frac{p}{q},$ see~\cite{jali, rafiams, rmt}. A similar matrix-faction description (MFD) holds for $ M \in \mathbb{M}(\Omega)^{m\times n}.$ In fact, $M(z)$ can be represented as $$ M(z) = N_R(z)D_R(z)^{-1} = D_L(z)^{-1} N_L(z),$$ where $ N_L, N_R, D_L, D_R$ are holomorphic and $D_L$ and $D_R$ are regular. We discuss only a right MFD as the case of a left MFD is similar.

\vone 

%
%
%
%\begin{definition}(Matrix-fraction description) Let $\M \in \M( \Omega, \C^{  m\times n})$.  A right Matrix-fraction description (MFD) of  $\M$ is a decomposition of $\M(z)$ of the form $$\M(z)= N(z)D(z)^{-1} \ \mathrm{such \ that} \ \mathrm{Poles}(\M) = \eig(D),$$ where $N\in \mathbb{H} \left( \Omega ,  \C^{m \times n}\right)$ and $D \in \mathbb{H}(\Omega, \C^{m \times m})$ with $D$ being regular.
%	A right MFD $\M(z)= N(z)D(z)^{-1}$ is called $\emph{right coprime}$ if $N$ and $ D$ are right coprime. 
%\end{definition}

 \begin{definition} [Right MFD] Let $M \in \mathbb{M}(\Omega)^{m\times n}$.  A right matrix-fraction description (MFD) of $M(z)$ is a decomposition of $M(z)$  of the form $M(z)= N(z)D(z)^{-1}$, where $ N \in \mathbb{H}(\Omega)^{m\times n}$ and $D\in \mathbb{H}(\Omega)^{n\times n}$ with $D(z)$ being regular. 
 	
 	 \von	
 	A right MFD  $M(z)= N(z)D(z)^{-1}$ is a said to be right coprime if $N(z)$ and $D(z)$ are right coprime, that is, if 
	$\rank \left[ \begin{array}{c}
		N(z) \\ D(z)
	\end{array}\right] = n$ for  $z \in \C.$
	
\end{definition}

\vone 

A right coprime MFD of $M(z)$ exists and can be deduced from the Smith-McMillan form of $M(z).$ Indeed, we have the following result.

\vone 
\begin{proposition}\label{rmfd}  Let $M \in \mathbb{M}(\Omega)^{m\times n}.$ Then there exist non-unique pair $( N, D) \in \mathbb{H}(\Omega)^{m\times n} \times  \mathbb{H}(\Omega)^{n\times n} $ such that $N(z)$ and $D(z)$ are right coprime, $D(z)$ is regular, and $$ M(z) = N(z) D(z)^{-1}.$$ 
\end{proposition}
 \begin{proof} Let $ \Sigma_M(z) = \diag( \phi_1(z)/\psi_1(z), \ldots, \phi_r(z)/\psi_r(z), 0_{m-r, n-r})$  be the Smith-McMillan form of $M(z).$ Then by Theorem~\ref{smc}, we have $ M(z) = E(z) \Sigma_M(z) F(z)$, where $ E$ and $F$ are invertible.   
 Define   $$ N_\phi(z): = \text{diag}(\phi_1(z),\cdots,\phi_r(z)) \oplus 0_{m-r, n-r} \text{ and }  D_\psi (z): = \text{diag}(\psi_1(z),...,\psi_r(z)) \oplus I_{n-r}.$$
Then $ \Sigma_M(z) = N_\phi(z) D_\psi(z)^{-1}$ is a right coprime MFD of $\Sigma_M(z).$ Since $ E$ and $F$ are invertible, setting $ N(z) := E(z) N_\phi(z)$ and $ D(z) := F(z)^{-1} D_\psi(z),$ it follows that $N(z)$ and $D(z)$ are right coprime and that $ M(z) = E(z)\Sigma_M(z) F(z) = N(z) D(z)^{-1}.$ Indeed, $$ \rank \bmatrix{ N(z) \\ D(z) } = \rank \bmatrix{ E(z) N_\phi(z)\\ F(z)^{-1} D_\psi(z)} = \rank \bmatrix{ N_\phi(z) \\ D_\psi(z)} = n$$ for all $ z \in \Omega.$ 
Hence $M(z) = N(z) D(z)^{-1}$ is a right coprime MFD. 
\end{proof}

\vone 
We now show that a right coprime MFD of $M(z)$ is unique up to a unit element.

\vone

\begin{theorem}\label{eqv:mfd}
	Let $M \in \mathbb{M}(\Omega)^{m\times n}$. If  $M(z)= N_1(z)D_1(z)^{-1} =  N_2(z)D_2(z)^{-1} $ are right coprime MFDs then there exists $ U \in \mathrm{GL}_n(\mathbb{H}(\Omega))$  such that $N_1(z)= N_2(z)U(z)$ and $D_1(z)= D_2(z)U(z).$

	\von In particular, if $D_1(z)$ and $D_2(z)$ are matrix polynomials then $U(z)$ is a unimodular matrix polynomial. 
\end{theorem}

\begin{proof} Note that   $ \wp_{\Omega}(M) = \sig_{\Omega} (D_1) = \sig_{\Omega}(D_2)$.
	Since $N_1(z)D_1(z)^{-1} =  N_2(z)D_2(z)^{-1} $, we have $N_1(z)= N_2(z)D_2(z)^{-1}D_1(z) = N_2(z) U(z),$ where $U(z):= D_2(z)^{-1}D_1(z).$ Then $ U \in  \mathbb{M}(\Omega)^{n\times n}$. We show that  $U \in \mathrm{GL}_n(\mathbb{H}(\Omega))$  by proving that $D_2(z)^{-1}D_1(z)$ and its inverse $D_1(z)^{-1}D_2(z)$ are both analytic in $\Omega.$  
	
	\von 
	
	Since $N_1(z)$ and $D_1(z)$ are right coprime, by Theorem~\ref{cprime2}(d), there exist analytic matrices $X \in \mathbb{H}(\Omega)^{n \times m}$ and $ Y \in \mathbb{H}( \Omega)^{n\times n}$ such that $X(z)N_1(z) + Y(z)D_1(z)= I_n$. 
	Inserting $N_1(z)= N_2(z)U(z) $ we have
	$$X(z)N_2(z) U(z)+ Y(z)D_1(z)= I_n \Longrightarrow [X(z)N_2(z) + Y(z)D_2(z)]U(z) = I_n.$$ This shows that $ U(z)^{-1} = X(z)N_2(z) + Y(z)D_2(z) $ is analytic in $\Omega.$  
	
	\von 
	Similarly, considering $V(z) :=D_1(z)^{-1}D_2(z) $ and interchanging the role of $D_1(z)$ and $D_2(z)$ in the above proof, it follows that $ V(z)^{-1} $ is analytic in $\Omega.$ Since $V(z)^{-1} =  U(z)$ and 
	 both $U(z)$ and $U(z)^{-1}$ are analytic in $\Omega,$ we have $ U \in \mathrm{GL}_n(\mathbb{H}(\Omega)).$ 
	This proves that $ N_1(z) = N_2(z) U(z)$ and $ D_1(z) = D_2(z) U(z).$ 
\von

Now, if $ D_1(z)$ and $D_2(z)$ are matrix polynomials, then $ U(z) = D_1(z)D_2(z)^{-1}$ is a rational matrix. Since both $U(z)$ and $U(z)^{-1}$ are analytic, both $U(z)$ and $U(z)^{-1}$ must be matrix polynomials. In other words, $U(z)$ is a unimodular matrix polynomial.  \end{proof}

\vone In view of the proof of Proposition~\ref{rmfd}, we have the following result.  \vone 

\begin{theorem} \label{rcmfd}
	Let $M \in \mathbb{M}(\Omega)^{m\times n}$. Let  $M(z) = N(z)D(z)^{-1}$ be  a right coprime MFD. Let $ \Sigma_M(z) := \diag(\phi_1(z)/\psi_1(z), \ldots, \phi_r(z)/\psi_r(z), 0_{m-r, n-r})	$ be the Smith-McMillan form of $ M(z).$ Then $ S_N(z) := \diag(\phi_1(z), \ldots, \phi_r(z), 0_{m-r, n-r})$ is the Smith form of $ N(z)$ and $ S_D(z) := I_{n-r} \oplus \diag(\psi_r(z), \psi_{r-1}(z), \ldots, \psi_1(z))$ is the Smith form of $D(z).$ 
	In particular, we have 	
	\begin{align*}
		\sigma_{\Omega} (M)&=  \sigma_{\Omega}(N), \ \  \mathrm{Ind}_e(\lambda, M) = \mathrm{Ind}_e(\lambda, N)\ \text{ for } \ \lambda \in  \sigma_{\Omega}(M), \\  
		\wp_{\Omega}(M)&=  \sigma_{\Omega}(D), \ \  \mathrm{Ind}_p(\lambda, M) = \mathrm{Ind}_e(\lambda, D)\ \text{ for } \ \lambda \in \wp_{\Omega}(M).
	\end{align*}
	
\end{theorem}

\vone 
\begin{proof}  The proof of Proposition~\ref{rmfd} shows that $$ 	M(z)= E(z)\Sigma_M(z) F(z)= E(z)N_\phi (z)D_\psi(z)^{-1} F(z) = N(z) D(z)^{-1}$$ are right coprime MFDs. Hence by Theorem~\ref{eqv:mfd}, there exists $U \in \mathrm{GL}_n(\mathbb{H}(\Omega))$ such that $  N(z) = E(z) N_{\phi}(z) U(z) $ and $ D(z) = F(z)^{-1}D_{\psi}(z) U(z).$  Consequently, $$ S_N(z) = S_{N_\phi}(z) = N_{\phi}(z) = \diag(\phi_1(z), \ldots, \phi_r(z), 0_{m-r, n-r}) $$ is the Smith form of $N(z)$ and $$ S_D(z) = S_{D_{\psi}}(z) =   \diag(I_{n-r}, \psi_r(z), \psi_{r-1}(z), \ldots, \psi_1(z))$$ is the Smith form of $D(z).$

	Obviously, the zeros and poles of $M(z)$ are given by the eigenvalues of  $N_\phi(z)$ and $D_\psi (z)$, respectively.  
	In fact, we have 
	$$ 	\sigma_{\Omega}(M)=  \sig_{\Omega}(N_\phi) = \sig_{\Omega}(N) \text{ and } 	\wp_{\Omega}(M)=  \sig_{\Omega}(D_\psi) = \sig_{\Omega}(D). $$
	Further, we have 	
\beano	
		\mathrm{Ind}_e(\lambda, M) &=& \mathrm{Ind}_e(\lambda, N_\phi) = \mathrm{Ind}_e(\lambda, N)\ \text{for} \ \lambda \in  \sigma(M), \\  
		\mathrm{Ind}_p(\lambda, M) &=& \mathrm{Ind}_e(\lambda, D_\psi) = \mathrm{Ind}_e(\lambda, D)\ \text{for} \ \lambda \in \wp_{\Omega}(M).
\eeano
\end{proof}

We mention that similar results hold for left coprime MFDs of $M(z).$   \vone

{\bf Least order:} The  least order of a rational matrix $G(z) \in \C(z)^{m\times n}$ plays an important role linear systems theory which is defined as follows; see~\cite[p.38]{vardulakis}. Let $G(z) = N(z) D(z)^{-1}$ be a right coprime MFD of $G(z)$, where $N(z) \in \C[z]^{m\times n}$ and $D(z) \in \C[z]^{n\times n}$ is regular. Then $ \widehat{\nu}(G) := \deg \det(D(z))$ is called the {\em least order} of $G(z).$  The least order $\widehat{\nu}(G)$ is independent of a particular choice of right coprime MFD of $G(z).$ 
If $G(z) = \widehat{N}(z) \widehat{D}(z)^{-1}$ is a right MFD then $\widehat{\nu}(G) \leq \deg \det (\widehat{D}(z)).$ Further, if  $ G(z) = D_L(z)^{-1} N_L(z)$ is a left coprime MFD then $\widehat{\nu}(G) = \deg \det( D_L(z)).$ Furthermore, if the Smith-McMillan form~\cite{kailath, vardulakis, rosenbrock70} of $G(z)$ is given by $$ G(z) = \phi_1(z)/\psi_1(z)\oplus \cdots \oplus \phi_r(z)/\psi_r(z) \oplus 0_{m-r, n-r}$$ then $\widehat{\nu}(G) = \deg \left( \prod^r_{j=1} \psi_j(z)\right).$ Finally, if $ G(z) = P(z) + G_{sp}(z),$ where $ P(z) \in \C[z]^{m\times n}$ and $ G_{sp}(z) \in \C(z)^{m\times n}$ is strictly proper, then the McMillan degree of $G(z)$ is defined as $\delta_M(G) := \widehat{\nu}(G_{sp}) + \widehat{\nu}( P(1/z))$. See~\cite[p.38-39]{vardulakis} for further details.

\vone

Now, we  extend the notion of least order of  a rational matrix to a meromorphic matrix. Let  $ M \in \mathbb{M}(\Omega)^{m\times n}$ and $M(z) = N(z) D(z)^{-1}$ be a right coprime MFD, where $ N \in  \mathbb{H}(\Omega)^{m\times n}$ and $ D \in  \mathbb{H}(\Omega)^{n\times n}$ is regular. Since $\deg \det (D(z))$ makes no sense unless $\det(D(z))$ is a  polynomial, to overcome the difficulty we introduce the notion of order of $D(z)$ which, in a sense, generalizes the notion of degree of scalar polynomial to a holomorphic function.  %to and replace $\deg \det (D(z))$ with the order of $D(z).$
 \von

\begin{definition}[order]\label{ord}  Let $ A \in \mathbb{H}(\Omega)^{n\times n}$ be regular. Let $\partial_A$ denote that principal divisor of $\det(A(z))$ on $\Omega.$ Then $\partial_A$ is called the order of $A(z)$ on $\Omega.$ 

 \end{definition}

\vone 
Observe that  if $ A \in \mathbb{H}(\Omega)^{n\times n}$ is regular then  $\partial_A = 0 \Longleftrightarrow A  \in \mathrm{GL}_n(\mathbb{H}(\Omega)).$  Further, if 
$ B \in \mathbb{H}(\Omega)^{n\times n}$ is regular then $ \partial_{AB} = \partial_A+\partial_B.$ Hence $$ \partial_{AB} = \partial_A \Longleftrightarrow \partial_B = 0 \Longleftrightarrow B \in \mathrm{GL}_n(\mathbb{H}(\Omega)).$$
Thus,  if $ A(z) \sim_{\Omega} B(z)$ then $ \partial_A = \partial_B.$ Note that $ \supp(\partial_A) = \sig_{\Omega}(A).$ Therefore, if $A(z)$ is a matrix polynomial and  $\sig_{\C}(A) = \{ \lam_1, \ldots, \lam_\ell\}$ then  $ \supp(\partial_A) =\{ \lam_1, \ldots, \lam_\ell\}$  and  $\partial_A(\lam_j)$ is the algebraic multiplicity of $\lam_j$ for $ j=1:\ell.$  Consequently, we have 
 $$ \deg \det(A(z)) = \partial_A(\lam_1) +\cdots+ \partial_A(\lam_\ell).$$ Thus, in a sense,  the order $\partial_A$ generalizes the notion of degree of the polynomial $\det(A(z))$ when $A(z)  $ is a  regular matrix polynomial to the holomorphic function $\det(A(z))$ when $A(z)$ is a regular holomorphic matrix.  

\von

\begin{definition}[least order] Let  $ M \in \mathbb{M}(\Omega)^{m\times n}$ and $ M(z) = N(z) D(z)^{-1}$ be a right coprime MFD, where $ N \in \mathbb{H}(\Omega)^{m\times n}$ and   $ D \in \mathbb{H}(\Omega)^{n\times n}$ is regular. Then $ \nu(M) := \partial_D$ is said to be the least order of $M(z)$, where $\partial_D$ is the order of $D(z).$ 
\end{definition}

\vone

Obviously,  $\nu(M)$ is independent of a particular choice of right coprime MFD of $M(z).$ Indeed, let  $M(z) = N_1(z) D_1(z)^{-1} =  N_2(z) D_2(z)^{-1}$ be  right coprime MFDs. Then by Theorem~\ref{eqv:mfd}, there exits $ U \in \mathrm{GL}_n(\mathbb{H}(\Omega))$  such that $ N_1(z) = N_2(z) U(z) $ and $ D_1(z)= D_2(z) U(z).$ Since $\partial_U = 0,$ we have $$\nu(M_1)= \partial_{D_1} = \partial_{D_2 U} = \partial_{D_2}+\partial_U = \partial_{D_2} = \nu(M_2).$$

Let $ M(z) = N(z) D(z)^{-1} $ be a right coprime MFD. Also, let $ M(z) = N_R(z) D_R(z)^{-1}$ be a right MFD.  If $N_R(z)$ and $D_R(z)$ are not right coprime then in view of Theorem~\ref{rcmfd}, we have $\wp_{\Omega}(M) = \sig_{\Omega}(D)  \subset \sig_{\Omega}(D_R)$ and $ \nu(M) = \partial_D \leq \partial_{D_R}.$   Further,  if $ M(z) = D_L(z)^{-1} N_L(z)  = N_R(z) D_R(z)^{-1}$  are left  and right coprime  MFDs then by Theorem~\ref{rcmfd}, we have  $ \nu(M) = \partial_{D_R} = \partial_{D_L}.$

If  $ M_1, M_2 \in \mathbb{M}(\Omega)^{m\times n}$ and $ M_1(z) \sim_{\Omega} M_2(z)$ then $ \nu(M_1) = \nu(M_2).$  Indeed, let $ M_1(z) = E(z) M_2(z) F(z),$ where $ E \in \mathrm{GL}_m(\mathbb{H}(\Omega))$ and $ F \in \mathrm{GL}_n(\mathbb{H}(\Omega)).$ Let $ M_2(z) = N(z) D(z)^{-1}$ be a right coprime MFD. Then $ \nu(M_2) = \partial_D.$ Now  $$ M_1(z) = E(z) N(z) D(z)^{-1} F(z) = (E(z)N(z)) ( F(z)^{-1}D(z))^{-1}$$ is a right coprime MFD of $ M_1(z).$ Hence   $ \nu(M_1) = \partial_{F^{-1}D} = \partial_D = \nu(M_2)$ since $ F^{-1} \in  \mathrm{GL}_n(\mathbb{H}(\Omega)).$ 
Next, let  $$ \Sigma_M (z) = \phi_1(z)/\psi_1(z)\oplus \cdots \oplus \phi_r(z)/\psi_r(z) \oplus 0_{m-r, n-r}$$ be the Smith-McMillan form of $M(z)$ and $\kappa_1 \geq \cdots \geq \kappa_r$ be the principal divisors of the invariant pole functions $\psi_1, \ldots, \psi_r$, respectively. Let $ \psi(z) := \prod^r_{j=1} \psi_j(z)$. Then by Theorem~\ref{rcmfd}, we have $ \nu(M) = \partial_{D} = \kappa_1 + \cdots + \kappa_r= \nu_{\psi},$ where $\nu_{\psi}$ is the principal divisor of $\psi(z)$ on $\Omega.$ In particular, if  $M(z)$ is a rational matrix with right coprime MFD $ M(z) = N(z) D(z)^{-1}$ and $\wp_{\C}(M) = \{ \lam_1, \ldots, \lam_\ell\}$ then  
$$ \widehat{\nu}(M) =\deg \det (D(z)) = \nu(M)(\lam_1)+\cdots + \nu(M)(\lam_\ell)  = \deg (\psi)$$ is the least order of the rational matrix $M(z).$   This proves the following result.

\vone \begin{theorem}\label{lord}  Let  $ M \in \mathbb{M}(\Omega)^{m\times n}$.    Then the following results hold. 
\begin{itemize}
\item[(a)] Let 	$M(z) = D_L(z)^{-1} N_L(z)  = N_R(z) D_R(z)^{-1}$ be  left and right coprime  MFDs, where $ N_L, N_R \in \mathbb{H}(\Omega)^{m\times n}$ and   $ (D_L, D_R) \in \mathbb{H}(\Omega)^{m\times m} \times \mathbb{H}(\Omega)^{n\times n}$. Then $ \nu(M) = \partial_{D_L} = \partial_{D_R}$. 

\item[(b)] $\supp(\nu(M)) = \wp_{\Omega}(M).$ 
\item[(c)] If $ M(z) = N(z) D(z)^{-1}$ is a right MFD then $ \nu(M) \leq \partial_D.$
\item[(d)] We have $ \nu(M) = \nu(UMV)$ for all $ U \in \mathrm{GL}_m(\mathbb{H}(\Omega))$ and $ V \in \mathrm{GL}_n(\mathbb{H}(\Omega)).$
\item[(e)] Let  $ \Sigma_M (z) = \phi_1(z)/\psi_1(z)\oplus \cdots \oplus \phi_r(z)/\psi_r(z) \oplus 0_{m-r, n-r}$ be the Smith-McMillan form of $M(z)$ and $\kappa_1,\ldots,  \kappa_r$ be the principal divisors of  $\psi_1, \ldots, \psi_r$, respectively. Let $ \psi(z) := \prod^r_{j=1} \psi_j(z)$.	Then $ \nu(M) =  \kappa_1 + \cdots + \kappa_r= \nu_{\psi},$ where $ \nu_{\psi}$ is the principal divisor of $\psi$ on $\Omega.$
\item[(f)] If  $M(z)$ is a rational matrix with right coprime MFD $ M(z) = N(z) D(z)^{-1}$ and $\wp_{\C}(M) = \{ \lam_1, \ldots, \lam_\ell\}$ then  
\beano \widehat{\nu}(M) &=& \deg \det (D(z)) = \nu(M)(\lam_1)+\cdots + \nu(M)(\lam_\ell) \\ & =& \sum^r_{j=1} (\kappa_1(\lam_j) + \cdots + \kappa_r(\lam_j)) = \deg (\psi).\eeano
\end{itemize} 
\end{theorem}

\vone At this point one may ask: Is it possible to extend the notion of McMillan degree of a rational matrix to a meromorphic matrix-valued function $ M : \C \longrightarrow \C^{m\times n}?$

 To answer the question, let us look closely at the McMillan degree of a rational matrix.  Let $ G(z) \in \C(z)^{m\times n}$. Then $G(z)$  is meromorphic on the Riemann sphere $\C_{\infty} := \C \cup\{\infty\}$ and the McMillan degree of $G(z)$ is the total count (including multiplicity) of finite and infinite poles of $G(z).$ 
 In contrast, a meromorphic matrix-valued function $ M : \C \longrightarrow \C^{m\times n}$ may not be meromorphic on $\C_{\infty}$ and can have an essential singularity at infinity. For instance, $M(z) := \diag( z, (z-1)/(e^z-1))$ is  meromorphic on $\C$ but not on $\C_{\infty}$. Indeed, we have $\wp_{\C}(M) =\{ 2n\pi i : n \in \mathbb{Z}\}$ and $M(z)$ has an essential singularity at infinity. A well known fact is that if $ f : \C \longrightarrow \C$ is meromorphic on $\C_{\infty}$ then $f$ is a rational function~\cite{alf}.  Hence, if $ M: \C\longrightarrow \C^{m\times n}$ is meromorphic on $\C_{\infty}$ then $M(z)$ is a rational matrix and in such a case the McMillan degree is already defined.  If $M(z)$ is a transcendental meromorphic matrix, that is, $M(z)$ is meromorphic on $\C$ but not a rational matrix, then either (a) $\wp_{\C}(M)$ is a finite set and $M(z)$ has an isolated essential singularity at infinity or (b) $\wp_{\C}(M)$ is an infinite set and there is a sequence $(\lam_n) $ in $\wp_{\C}(M)$ such that $ \lam_n \longrightarrow \infty$ and $ n \rightarrow \infty$   in which case $M(z)$ has a non-isolated essential singularity at infinity. Thus, for a transcendental meromorphic matrix $M(z)$, it is not clear how to define  an analogue of McMillan degree. Note that the transfer function associated with a TDS such as (\ref{tds}) is  a transcendental meromorphic matrix. So, it is an open problem to define an analogue of McMillan degree for the transfer function of a TDS and analyze its implications on the TDS. 
	
%Note that $G(z)$ can be uniquely written as   $G(z) = P(z) + G_{sp}(z) \in \C(z)^{m\times n}$,  where $G_{sp}(z) \in \C(z)^{m\times n}$  is strictly proper and $ P(z) \in \C[z]^{m\times n}.$ The poles of $G(z)$ at infinity are the poles of $P(z)$ at infinity which are given by the poles of $P(1/z)$ at $ 0.$  Hence  $\delta_M(G) := \widehat{\nu}(G_{sp}) + \widehat{\nu}( P(1/z))$ is the McMillan degree of $G(z)$.

\vone
{\bf Argument principle:} We end this section with a proof of the argument principle for a regular meromorphic matrix-valued function using MFD.
The argument principle is a well known classical result in complex analysis~\cite[p.282]{churchil}. Let $f:\Omega \to \C$ be a meromorphic and $ \Gamma \subset \Omega$ be a rectifiable simple closed curve such that $\mathrm{Int}(\Gamma) \subset \Omega.$ Here $\mathrm{Int}(\Gamma)$ denotes the interior of the region enclosed by $\Gamma.$ 	
	Suppose that $f$ is analytic and nonzero on $\Gamma.$  	Let  $ P$ be the number of poles  and $ N$ be the number  of zeros of $f$ inside $\Gamma$ counting multiplicities.  Then~\cite[p.282]{churchil}  
	$$\frac{1}{2 \pi i} \int_\Gamma \frac{f'(z)}{f(z)}dz = N-P.$$

\von 
We now use MFD to provide a proof of the argument principle for a meromorphic matrix-valued function. See~\cite{gs} for argument principle for holomorphic operator-valued functions.\vone

\begin{theorem}[argument principle] \label{argp}
	Let $M \in \mathbb{M}(\Omega)^{n \times n}$ be regular. Let $ \Gamma \subset \Omega$ be a rectifiable simple closed curve such that $\mathrm{Int}(\Gamma) \subset \Omega.$ Suppose that $M$ is analytic on $\Gamma$ and  that $\Gamma$ does not pass through a zero of $M(z).$  Let   $P$ be the  number of poles (counting multiplicities) and $ N $  be the number of zeros (counting multiplicities) of $M(z)$ inside $\Gamma.$  Then
	$$\frac{1}{2 \pi i} \int_{ \Gamma} \mathrm{Tr} \left(M(z)^{-1}M'(z)\right) dz = N-P,$$ where $M'(z)$ is the derivative of $M(z)$ with respect to $z$. 
\end{theorem}

\vone 
\begin{proof}
	Let $M(z) = N(z)D(z)^{-1}$ be a right coprime MFD.  By Theorem  \ref{rcmfd}, the number of eigenvalues of $N(z)$ inside $\Gamma  =$ the number of zeroes of $M(z)$ inside $\Gamma = N$. Further,  the number of  eigenvalues  of $D(z)$ inside $\Gamma =$ the number of poles of $M(z)$ inside $\Gamma = P$. 

\von

Now, differentiating $M(z)$ with respect to $z$, we have  
	\begin{align*}
		M'(z) & =  N'(z)D(z)^{-1}+N(z)(D(z)^{-1})' \\ & =   N'(z)D(z)^{-1}+N(z) \left( -D(z)^{-1}D'(z)D(z)^{-1}\right).
	\end{align*}
	Hence  
	\begin{align*}
		\mathrm{Tr}\left( (M(z)^{-1}M'(z) \right)dz  & =\mathrm{Tr}\left(D(z)N(z)^{-1} \left(N'(z)D(z)^{-1}- N(z) D(z)^{-1}D'(z)D(z)^{-1} \right)\right)
		\\&= \mathrm{Tr}\left(D(z)N(z)^{-1}N'(z)D(z)^{-1}\right) - \mathrm{Tr}\left(D'(z)(D(z)^{-1}) \right) 
		\\& = \mathrm{Tr}\left(N(z)^{-1}N'(z)\right) - \mathrm{Tr}\left(D'(z)(D(z)^{-1}) \right).
	\end{align*}

\noin 	Consequently, we have 
	\begin{align*}
		\frac{1}{2 \pi i} \int_{\Gamma}\mathrm{Tr}\left( (M(z)^{-1}M'(z) \right)dz &=
		\frac{1}{2 \pi i} \int_{\Gamma} \mathrm{Tr}\left(N(z)^{-1}N(z)'\right) - \frac{1}{2 \pi i} \int_{\partial \Omega}\mathrm{Tr}\left(D(z)'(D(z)^{-1}) \right)dz \\ &= N-P.
	\end{align*}

The last equality follows from the argument principle for the complex function $ f(z) := \det(N(z))$ and the Jacobi formula~\cite{horn} $ f'(z)/f(z) = \mathrm{Tr}( N(z)^{-1}N'(z))$.  	
\end{proof}

\section{Analytic matrix description (AMD)} 
We have seen in (\ref{tds}) that a TDS gives an AMD  as well as  the associated transfer function. On the other hand, we now show that a meromorphic matrix can be represented as the transfer function of an AMD. Our aim in this section is to undertake a comprehensive analysis of AMDs,   transfer functions and their canonical forms.  \vone

We first discuss transfer function realization of a meromorphic matrix and an AMD  associated with such a realization.  The idea is to analyze zeros and poles of a meromorphic matrix via  zeros of appropriate holomorphic matrices. For instance, a function $ f \in \mathbb{M}(\Omega)$ when represented as $ f = p/q$, where $ p, q \in \mathbb{H}(\Omega)$ are coprime, can be studied by analyzing the matrix $H(z) := \bmatrix{ q(z)  & 1 \\ -p(z) & 0}.$ Indeed, $f(z)$ is the Schur complement of $q(z)$ in $H(z)$ and $\det H(z) = p(z) = q(z) f(z).$ This shows that $$  \sig_{\Omega}(f) = \sig_{\Omega}(H) \text{ and } \wp_{\Omega}(f) = \sig_{\Omega}(q).$$ We generalize the same idea to meromorphic matrices.

\von Realization theory for rational matrices is a well developed subject and has been studied extensively~\cite{kailath, vardulakis, rosenbrock70}. We now define realization of a meromorphic matrix. \von 

\begin{definition} \label{real} Let $M \in \mathbb{M}(\Omega)^{m\times n}.$ Then a representation of $M(z)$ of the form $$  M(z) = D(z) + C(z) A(z)^{-1}B(z)$$ is called a  realization of $M(z),$ where $(C, A, B) \in   \mathbb{H}(\Omega)^{m\times r} \times \mathbb{H}(\Omega)^{r\times r} \times \mathbb{H}(\Omega)^{r\times n}$ with  $A(z)$ being regular and $D \in \mathbb{H}(\Omega)^{m\times n}.$ A realization $( C, A, B, D)$ of $M(z)$  is called minimal if  $A(z), C(z)$ are right coprime and $A(z), B(z)$ are left coprime, that is, if $$ \rank \bmatrix{ A(z) \\ C(z)} = \rank \bmatrix{ A(z) & B(z)} = r \text{ for all } z \in \Omega.$$ 
	
\end{definition} 	
	
We now associate a holomorphic matrix  with a realization of $M(z)$ which can be gainfully utilized for analyzing zeros and poles of $M(z)$.  

\vone

\begin{definition}[AMD] \label{amd} Let $(C, A, B) \in   \mathbb{H}(\Omega)^{m\times r} \times \mathbb{H}(\Omega)^{r\times r} \times \mathbb{H}(\Omega)^{r\times n}$ with $A(z)$ being regular and $D \in \mathbb{H}(\Omega)^{m\times n}.$   A holomorphic matrix $\H  \in  \mathbb{H}(\Omega)^{(r+m)\times (r+n)}$ given by 
		$$\H(z):= \left[	\begin{array}{c|c}
		A(z) & B(z) \\
		\hline
		-C(z) &   D(z)
	\end{array}  \right]$$ is called an analytic matrix description (AMD) associated with the  transfer function $ M(z) := D(z) + C(z) A(z)^{-1}B(z)$. The matrix  $A(z)$ is called the state matrix of the AMD. 
The AMD $\H(z)$ is said to be irreducible if $A(z), C(z)$ are right coprime and $A(z), B(z)$ are left coprime. %, that is,   if $$ \rank \bmatrix{ C(z) \\ A(z)} = \rank \bmatrix{ B(z) & A(z)} = r \text{ for all } z \in \Omega.$$ 
%The transfer function $\M(z)$ is called minimal when $\H(z)$ is irreducible.       
\end{definition}

\vone

Observe that the transfer function $M(z)$ is the Schur complement of $A(z)$ in $\H(z)$. Also note that $\H(z)$ is irreducible $\Longleftrightarrow (C, A, B, D)$ is a minimal realization of $M(z).$  
We mention that in the case of a rational matrix $G(z) \in \C(z)^{m\times n}$, the matrix polynomial $\mathbf{P}(z)$ in (\ref{pmd}) is referred to as {\em Rosenbrock system matrix} and  the quadruple $[A(z), B(z), C(z), D(z)] \in \C[z]^{r\times r} \times \C[z]^{r\times n} \times \C[z]^{m\times r} \times \C[z]^{m\times n}$ is referred to as a PMD (polynomial matrix description) and $A(z)$ is called the state matrix of the PMD, see~\cite{kailath, vardulakis, rosenbrock70}. See also \cite{rafinami1,DAS2019}. In this paper, we make no such distinction and refer to the Rosenbrock system matrix $\mathbf{P}(z)$ as a PMD.   
\vone

\begin{remark} \label{system} We mention that one can define multiple system matrices corresponding to a realization of a meromorphic matrix. Indeed, let $( C, A, B, D)$ be a realization of $M \in \mathbb{M}(\Omega)^{m\times n}.$ Then 
$$  \left[	\begin{array}{c|c}
D(z) & -C(z) \\
\hline
B(z) &   A(z)
\end{array}  \right] \;\; \text{ and } \;\;  \left[	\begin{array}{c|c}
D(z) & C(z) \\
\hline
-B(z) &   A(z)
\end{array}  \right] $$ are AMDs  with transfer function  $M(z) = D(z)+ C(z)A(z)^{-1}B(z).$  Further,	
	$$  \left[	\begin{array}{c|c}
		A(z) & -B(z) \\
		\hline
		C(z) &   D(z)
	\end{array}  \right] \text{ and } \left[	\begin{array}{c|c}
	-A(z) & B(z) \\
	\hline
	C(z) &   D(z)
	\end{array}  \right]  $$ are AMDs with the same transfer function. 
	
\end{remark}

\vone 

The existence of an AMD  of a meromorphic matrix follows from an MFD. \vone

\begin{theorem} Let  $ M \in \mathbb{M}(\Omega)^{m\times n}.$ Then there exists an  AMD $\H $ given by 
		$$\H(z):= \left[	\begin{array}{c|c}
		A(z) & B(z) \\
		\hline
		-C(z) &   D(z)
	\end{array}  \right] \in \mathbb{H}(\Omega)^{(r+m)\times (r+n)}$$ for some $ r \in \mathbb{N}$	such that  $ M(z) = D(z) + C(z) (A(z))^{-1} B(z)$. \end{theorem}

\vone\begin{proof} Let $ M(z) = N(z) D(z)^{-1} $ be an MFD. Then $ \H(z) :=  \left[
	\begin{array}{c|c} D(z)  &   I_n\\
		\hline 
		- N(z) &   	0_{m\times n} 
	\end{array}
	\right]$ is an AMD with the transfer function $M(z) = N(z) D(z)^{-1}.$ 
\end{proof}

\vone Rosenbrock~\cite[p.52]{rosenbrock70} introduced the concept of {\em strict system equivalence}  in order to deal with multitude of PMDs associated with a rational transfer function. See also \cite[p.58]{vardulakis}. The  strict system equivalence is referred to as {\em Rosenbrock's system equivalence} (rse) in \cite[Ch.~8]{kailath}.

\begin{definition}[Rosenbrock's system equivalence, \cite{rosenbrock70}] 	Let $\mathbf{P}_1(z)$ and $\mathbf{P}_2(z)$ be PMDs   given by 
	$$ {\small \mathbf{P}(z):= \left[
		\begin{array}{c|c}
			A_1(z) & B_1(z) \\
			\hline
			-C_1(z) &   D_1(z)
		\end{array}  \right]_{(r+m) \times (r+n)} \emph{and} \  \  \mathbf{P}_2(z):= \left[
		\begin{array}{c|c}
			A_2(z) & B_2(z) \\
			\hline
			-C_2(z) &   D_2(z)
		\end{array}  \right]_{(\ell+m) \times (\ell+n)}}$$
	with state matrices $A_1 \in \C[z]^{r \times r} \emph{and} \ A_2 \in \C[z]^{\ell \times \ell}.$ Let $p \geq \max(r,\ell)$. Then $\mathbf{P}_1(z)$ and $\mathbf{P}_2(z)$ are said to be Rosenbrock system equivalent and written as $\mathbf{P}_1(z)\sim_{rse} \mathbf{P}_2(z)$ if there exist unimodular   $M, N \in \C[z]^{p\times p},$ $X \in \C[z]^{m \times p}$ and $Y \in \C[z]^{p \times n}$  such that 
	$$\left[
	\begin{array}{c|c}
		M(z)
		& 0 \\
		\hline
		X(z) &  I_m 
	\end{array}  \right] \left[
	\begin{array}{c|c}
		I_{p-r} & 0 \\
		\hline
		0&  \mathbf{P}_1(z)
	\end{array}  \right]\left[
	\begin{array}{c|c}
		N(z) & Y(z)\\
		\hline
		0 &  I_n
	\end{array}  \right] =  \left[
	\begin{array}{c|c}
		I_{p-\ell}  & 0 \\
		\hline
		0& \mathbf{P}_2(z)
	\end{array}  \right].$$
\end{definition}

We now extend Rosenbrock's system equivalence of PMDs to AMDs and continue to  refer to the equivalence as Rosenbrock's system equivalence (rse) of AMDs.  %If $\mathbf{H}_1$ and $ \mathbf{H}_2$ are AMDs then we write $\mathbf{H}_1 \sim_{rse} \mathbf{H}_2$ to mean that $\mathbf{H}_1$ and $ \mathbf{H}_2$ are Rosenbrock system equivalent. 

\vone

\begin{definition}\label{sse} 	Let $\H_1(z)$ and $\H_2(z)$ be AMDs   given by 
	$$ {\small \H_1(z):= \left[
\begin{array}{c|c}
	A_1(z) & B_1(z) \\
	\hline
	-C_1(z) &   D_1(z)
\end{array}  \right]_{(r+m) \times (r+n)} \emph{and} \  \  \H_2(z):= \left[
\begin{array}{c|c}
	A_2(z) & B_2(z) \\
	\hline
	-C_2(z) &   D_2(z)
\end{array}  \right]_{(\ell+m) \times (\ell+n)}}$$
with state matrices $A_1 \in \mathbb{H}(\Omega)^{r \times r} \emph{and} \ A_2 \in \mathbb{H}(\Omega)^{\ell \times \ell}.$ Let $p \geq \max(r,\ell)$. Then $\H_1(z)$ and $\H_2(z)$ are said to be Rosenbrock system equivalent and written as $\H_1(z)\sim_{rse} \H_2(z)$ if there exist  $M, N \in \mathrm{GL}_p(\mathbb{H}(\Omega)),$ $X \in \mathbb{H}(\Omega)^{m \times p}$ and $Y \in \mathbb{H}(\Omega)^{p \times n}$  such that 
$$\left[
\begin{array}{c|c}
	M(z)
	& 0 \\
	\hline
	X(z) &  I_m 
\end{array}  \right] \left[
\begin{array}{c|c}
	I_{p-r} & 0 \\
	\hline
	0&  \H_1(z)
\end{array}  \right]\left[
\begin{array}{c|c}
	N(z) & Y(z)\\
	\hline
	0 &  I_n
\end{array}  \right] =  \left[
\begin{array}{c|c}
	I_{p-\ell}  & 0 \\
	\hline
	0& \H_2(z)
\end{array}  \right].$$
\end{definition}

\vone 
%\begin{remark} \label{sse2}
%We mention that for  permuted analytic system matrices, strict system equivalence is defined as follows. Let the system matrices $\H_1(z)$ and $\H_2(z)$ be given by 
%	$$ \H_1(z):= \left[
%\begin{array}{c|c}
%	A_1(z) & B_1(z) \\
%	\hline
%	-C_1(z) &   D_1(z)
%\end{array}  \right]_{(r+m) \times (r+n)} \emph{and} \  \  \H_2(z):= \left[
%\begin{array}{c|c}
%	A_2(z) & B_2(z) \\
%	\hline
%	-C_2(z) &   D_2(z)
%\end{array}  \right]_{(\ell+m) \times (\ell+n)}$$
%with state matrices $A_1 \in \mathbb{H}(\Omega)^{r \times r} \emph{and} \ A_2 \in \mathbb{H}(\Omega)^{\ell \times \ell}.$ Let $p \geq \max(r,\ell)$. Then $\H_1(z)\sim_{sse} \H_2(z)$ 
%if there exist  $M, N \in \mathrm{GL}_p(\mathbb{H}(\Omega)),$ $X \in \mathbb{H}(\Omega)^{m \times p}$ and $Y \in \mathbb{H}(\Omega)^{p \times n}$  such that 
%$$\left[
%\begin{array}{c|c}
%	M(z)
%	& 0 \\
%	\hline
%	X(z) &  I_m 
%\end{array}  \right] \left[
%\begin{array}{c|c}
%	 I_{p-r} & 0 \\
%	\hline
%	0&  \H_1(z)
%\end{array}  \right]\left[
%\begin{array}{c|c}
%	 N(z) & Y(z)\\
%	\hline
%	0 &  I_n
%\end{array}  \right] =  \left[
%\begin{array}{c|c}
%	  I_{p-\ell}  & 0 \\
%	\hline
%	0& \H_2(z)
%\end{array}  \right].$$
%
%\end{remark} 

%\vone The strict system equivalence defined above reduces to the strict system equivalence introduced by Rosenbrock~(see, \cite[p.52]{rosenbrock70} and \cite[p.58]{vardulakis}) when $\H_1(z)$ and $\H_2(z)$ are matrix polynomials. In such a case, the polynomial system matrices $\H_1(z)$ and $\H_2(z)$ are called Rosenbrock system matrices. We refer to a Rosenbrock system matrix as a polynomial system matrix.   

It is well known~\cite{rosenbrock70, kailath, vardulakis} that Rosenbrock's system equivalence of a PMD preserves the transfer function. It turns out that the same is true for an AMD.  

\vone \begin{theorem} \label{th:sse} Let $ \H_1(z)$ and $\H_2(z)$ be AMDs given by 	$${\small  \H_1(z):= \left[
	\begin{array}{c|c}
		A_1(z) & B_1(z) \\
		\hline
		-C_1(z) &   D_1(z)
	\end{array}  \right]_{(r+m) \times (r+n)} \emph{and} \  \  \H_2(z):= \left[
	\begin{array}{c|c}
		A_2(z) & B_2(z) \\
		\hline
		-C_2(z) &   D_2(z)
	\end{array}  \right]_{(\ell+m) \times (\ell+n)}}$$
	with state matrices $A_1 \in \mathbb{H}(\Omega)^{r \times r} \emph{and} \ A_2 \in \mathbb{H}(\Omega)^{\ell \times \ell}.$ 
	 If $ \H_1(z) \sim_{rse} \H_2(z)$ then $$M_1(z) = D_1(z) + C_1(z)A_1(z)^{-1}B_1(z) = D_2(z) + C_2(z)A_2(z)^{-1}B_2(z) =M_2(z).$$  %Further, $ \partial_{A_1}= \partial_{A_2}.$
\end{theorem} 	
	
\begin{proof} The proof is purely computational and follows from the same arguments as given in the proof of \cite[Theorem~2.13]{vardulakis} keeping in mind that unimodular matrix polynomials are replaced with invertible holomorphic matrices (unit elements) and matrix polynomials are replaced with holomorphic matrices. \end{proof}	

%
%\vone \begin{remark} \label{rem:sse} Let $\H \in   \mathbb{H}(\Omega)^{(m+r)\times (n+r)}$ be an analytic system matrix given by  	
%	$$\H(z):= \left[	\begin{array}{c|c}
%		D(z) & C(z) \\
%		\hline
%		B(z) &  - A(z)
%	\end{array}  \right], \;\text{ where } \; A(z) \text{ is regular}.$$ Then it follows that the following are invariant under strict system equivalence of $ \H(z):$ 
%	\begin{itemize} \item[(a)] The Smith form of $ \H(z)$. \item[(b)] The Smith form of $A(z).$ \item[(c)] The Smith form of $\bmatrix{B(z) & -A(z)}.$ \item[(d)] The Smith form of $ \bmatrix{ C(z) \\ -A(z)}.$ 
%	\end{itemize} 
%\end{remark}

\vone 

The Rosenbrock's system equivalence of irreducible PMDs can be characterized by their transfer functions. \von 

\begin{theorem}[Rosenbrock, \cite{rosenbrock70}]\label{equivlnt}
 Let $\mathbf{P}_1(z)$ and $\mathbf{P}_2(z)$ be  PMDs with state matrices $A_1(z)$ and $A_2(z)$, respectively.  If $\mathbf{P}_1(z) \sim_{rse} \mathbf{P}_2(z) $ then $\mathbf{P}_1(z)$ and $\mathbf{P}_2(z)$ have the same transfer function and $\deg \det( A_1(z)) = \deg \det(A_2(z)).$

 Further, if $\mathbf{P}_1(z)$ and $\mathbf{P}_2(z)$  are irreducible then $\mathbf{P}_1(z) \sim_{rse} \mathbf{P}_2(z) $ if and only if $\mathbf{P}_1(z)$ and $\mathbf{P}_2(z)$ have the same transfer function. 
\end{theorem}

\vone 

We show that an analogue of Theorem~\ref{equivlnt}  holds for AMDs. For this purpose, we need to extend {\em Fuhrmann's system equivalence} (fse)~\cite[p.67]{vardulakis} of PMDs to  AMDs.  
	
	\vone 

\begin{definition}[Fuhrmann's system equivalence for PMD, \cite{vardulakis}] 	Let $\mathbf{P}_1(z)$ and $\mathbf{P}_2(z)$ be PMDs   given by 
	$$ {\small \mathbf{P}(z):= \left[
		\begin{array}{c|c}
			A_1(z) & B_1(z) \\
			\hline
			-C_1(z) &   D_1(z)
		\end{array}  \right]_{(r+m) \times (r+n)} \emph{and} \  \  \mathbf{P}_2(z):= \left[
		\begin{array}{c|c}
			A_2(z) & B_2(z) \\
			\hline
			-C_2(z) &   D_2(z)
		\end{array}  \right]_{(\ell+m) \times (\ell+n)}}$$
	with state matrices $A_1 \in \C[z]^{r \times r} \emph{and} \ A_2 \in \C[z]^{\ell \times \ell}.$  Then $\mathbf{P}_1(z)$ and $\mathbf{P}_2(z)$ are said to be Fuhrmann system equivalent and written as $\mathbf{P}_1(z)\sim_{fse} \mathbf{P}_2(z)$ if there exist $M \in \C[z]^{\ell\times r}, X \in \C[z]^{m\times r}$  and $ N \in \C[z]^{\ell\times r}, Y \in \C[z]^{\ell\times n} $ such that 
	$$\left[
	\begin{array}{c|c}
		M(z) 	& 0_{\ell\times m} \\
		\hline
		X(z) &   I_m
	\end{array}  \right] \left[ \begin{array}{c|c}
		A_1(z) & B_1(z) \\
		\hline
		-C_1(z) &   D_1(z)
	\end{array}  \right] 
	= \left[
	\begin{array}{c|c}
		A_2(z) & B_2(z) \\
		\hline
		-C_2(z) &   D_2(z)
	\end{array}  \right] \left[
	\begin{array}{c|c}
		N(z) & Y(z)\\
		\hline
		0_{n\times r} &   I_n
	\end{array}  \right]$$ and the following conditions hold: 
	\begin{itemize} \item[(a)] $M(z)$ and $ A_2(z)$ are left coprime. \item[(b)] $A_1(z) $ and $N(z)$ are right coprime. \end{itemize} 
	
\end{definition}
\vone

We now define Fuhrmann's system equivalence of AMDs and  refer to the equivalence as {\em Fuhrmann's system equivalence} (fse) of AMDs. \\

\begin{definition}\label{fse} Let $\H_1(z)$ and $\H_2(z)$ be AMDs   given by 
	$$ {\small \H_1(z):= \left[
	\begin{array}{c|c}
		A_1(z) & B_1(z) \\
		\hline
		-C_1(z) &   D_1(z)
	\end{array}  \right]_{(r+m) \times (r+n)} \emph{and} \  \  \H_2(z):= \left[
	\begin{array}{c|c}
		A_2(z) & B_2(z) \\
		\hline
		-C_2(z) &   D_2(z)
	\end{array}  \right]_{(\ell+m) \times (\ell+n)}}$$
	with state matrices $A_1 \in \mathbb{H}(\Omega)^{r \times r} \emph{and} \ A_2 \in \mathbb{H}(\Omega)^{\ell \times \ell}$.  Then $\H_1(z)$ and $\H_2(z)$ are said to be Fuhrmann system equivalent and written as $\H_1(z)\sim_{fse} \H_2(z)$ if there exist  $M \in \mathbb{H}(\Omega)^{\ell\times r}, X \in \mathbb{H}(\Omega)^{m\times r}$  and $ N \in \mathbb{H}(\Omega)^{\ell\times r}, Y \in \mathbb{H}(\Omega)^{\ell\times n} $ such that 
	$$\left[
	\begin{array}{c|c}
		M(z) 	& 0_{\ell\times m} \\
		\hline
		X(z) &   I_m
	\end{array}  \right] \left[ \begin{array}{c|c}
	A_1(z) & B_1(z) \\
	\hline
	-C_1(z) &   D_1(z)
	\end{array}  \right] 
	 = \left[
	\begin{array}{c|c}
	A_2(z) & B_2(z) \\
	\hline
	-C_2(z) &   D_2(z)
	\end{array}  \right] \left[
	\begin{array}{c|c}
	N(z) & Y(z)\\
	\hline
	0_{n\times r} &   I_n
	\end{array}  \right]$$ and the following conditions hold: 
	\begin{itemize} \item[(a)] $M(z)$ and $ A_2(z)$ are left coprime. \item[(b)] $A_1(z) $ and $N(z)$ are right coprime. \end{itemize} 
 
\end{definition}

\vone Notice that the matrices $M(z)$ and $N(z)$ are non-square when $ r \neq \ell$. Even when $r =\ell$, unlike in the case of Rosenbrock's system equivalence,  the matrices $M(z)$ and $N(z)$ need not be unit elements of $\mathbb{H}(\Omega)^{r\times r}$. 
It turns out that, as in the case of PMDs~(see,\cite[Theorem~2.30]{vardulakis}), Fuhrmann's system equivalence and Rosenbrock's system equivalence are identical equivalence relations on AMDs. For completeness, we provide a proof as appropriate modifications  are needed in the proof  for AMDs. 

\vone 

\begin{theorem} \label{sse:fse} Let $\H_1(z)$ and $\H_2(z)$ be AMDs   given by 
	$$ {\small  \H_1(z):= \left[
	\begin{array}{c|c}
		A_1(z) & B_1(z) \\
		\hline
		-C_1(z) &   D_1(z)
	\end{array}  \right]_{(r+m) \times (r+n)} \emph{and} \  \  \H_2(z):= \left[
	\begin{array}{c|c}
		A_2(z) & B_2(z) \\
		\hline
		-C_2(z) &   D_2(z)
	\end{array}  \right]_{(\ell+m) \times (\ell+n)}}$$
	with state matrices $A_1 \in \mathbb{H}(\Omega)^{r \times r} \emph{and} \ A_2 \in \mathbb{H}(\Omega)^{\ell \times \ell}$. Then $$ \H_1(z) \sim_{fse} \H_2(z) \Longleftrightarrow \H_1(z) \sim_{rse} \H_2(z).$$

\end{theorem}
\vone 

\begin{proof} 	
Suppose that $\H_1(z) \sim_{fse}\H_2(z).$ 	By Definition~\ref{fse}, we have $ M(z)A_1(z) = A_2(z)N(z).$ Further, since $ A_2$ and $ M(z) $ are left coprime,  by Theorem~\ref{cprime2} there exist $\hat{X} \in \mathbb{H}(\Omega)^{\ell\times \ell}$ and $  \hat{Y} \in \mathbb{H}(\Omega)^{r\times \ell}$ such that $ A_2(z) \hat{X}(z) + M(z) \hat{Y}(z) = I_{\ell}.$ Similarly, since $A_1(z)$ and $ N(z)$ are right coprime, there exist $ \tilde{X} \in \mathbb{H}(\Omega)^{r\times \ell} $ and $ \tilde{Y} \in \mathbb{H}(\Omega)^{r\times r}$ such that $ \tilde{X}(z) N(z) + \tilde{Y}(z) A_1(z) = I_r.$ Consequently, we have $$ \bmatrix{ -\tilde{X}(z) & \tilde{Y}(z) \\ A_2(z) & M(z) } \bmatrix{ -N(z) & \hat{X}(z) \\ A_1(z) & \hat{Y}(z)} = \bmatrix{ I_r & W(z) \\ 0 & I_{\ell}}, $$ where $ W(z) := -\tilde{X}(z)\hat{X}(z) + \tilde{Y}(z)\hat{Y}(z).$  Since $\bmatrix{ I_r & -W(z) \\ 0 & I_\ell} \bmatrix{I_r & W(z) \\ 0 & I_\ell} = \bmatrix{ I_r & 0 \\ 0 & I_\ell},$  it follows that $\tilde{X}(z)$ and $\tilde{Y}(z)$ can be chosen such that $W(z) = 0.$ Indeed, setting $ \tilde{X}(z) \leftarrow -(\tilde{X}(z)+ W(z) A_2(z))$ and $ \tilde{Y}(z) \leftarrow \tilde{Y}(z) - W(z) M(z),$ we have 
\be\label{unit1} \bmatrix{ -\tilde{X}(z) & \tilde{Y}(z) \\ A_2(z) & M(z) } \bmatrix{ -N(z) & \hat{X}(z) \\ A_1(z) & \hat{Y}(z)} = \bmatrix{ I_r & 0 \\ 0 & I_{\ell}}\ee which shows that the block matrices in (\ref{unit1}) are unit elements, that is,
the block matrices are holomorphic and invertible. 

\vone Now the equality in Definition~\ref{fse} can be rewritten as

\beano  && \underbrace{\left[\begin{array}{cc|c} -\tilde{X}(z) & \tilde{Y}(z) & 0 \\ A_2(z) & M(z) & 0 \\ \hline -C_2(z) & X(z) & I_m \end{array}\right]}_{\mathcal{X}(z)}  \underbrace{\left[\begin{array}{cc|c} I_\ell & 0 & 0 \\  0 & A_1(z) & B_1(z)  \\ \hline 0 & -C_1(z) & D_1(z) \end{array}\right]}_{I_\ell \oplus \H_1(z)}
 =\\ &&   \underbrace{\left[\begin{array}{cc|c} I_r & 0 & 0 \\  0 & A_2(z) & B_2(z)  \\ \hline 0 & -C_2(z) & D_2(z) \end{array}\right]}_{I_r \oplus \H_2(z)}  \underbrace{\left[\begin{array}{cc|c} -\tilde{X}(z) & \tilde{Y}(z)A_1(z) & \tilde{Y}(z) B_(z) \\ I_\ell & N(z) & Y(z) \\ \hline 0 & 0 & I_n \end{array}\right]}_{\mathcal{Y}(z)}. 
\eeano
Hence it follows that $\H_1(z)\sim_{rse} \H_2(z)$ provided that $\mathcal{X}(z)$ and $\mathcal{Y}(z)$ are unit elements. It follows from (\ref{unit1}) that $\mathcal{X}(z)$ is holomorphic and invertible, that is, $\mathcal{X}(z)$ a unit element. On the other hand, \be\label{unit2} \bmatrix{ I_{r} & \tilde{X}(z) \\ 0 & I_\ell} \bmatrix{ -\tilde{X}(z) & \tilde{Y}(z) A_1(z) \\ I_r & N(z)} = \bmatrix{ 0 & I_r \\ I_\ell & N(z)}\ee shows that the block matrices in (\ref{unit2}) are holomorphic and invertible which in turn show that $\mathcal{Y}(z)$ is holomorphic and invertible, that is, $\mathcal{Y}(z)$ is a unit element.  

\vone

Conversely, suppose that $ \H_1(z) \sim_{rse} \H_2(z).$ Then  there exist  $M, N \in \mathrm{GL}_p(\mathbb{H}(\Omega)),$ $X \in \mathbb{H}(\Omega)^{m \times p}$ and $Y \in \mathbb{H}(\Omega)^{p \times n}$  such that 
$$\left[
\begin{array}{c|c}
	M(z)
	& 0 \\
	\hline
	X(z) &  I_m 
\end{array}  \right] \left[
\begin{array}{c|c}
	I_{p-r} & 0 \\
	\hline
	0&  \H_1(z)
\end{array}  \right] =  \left[
\begin{array}{c|c}
	I_{p-\ell}  & 0 \\
	\hline
	0& \H_2(z)
\end{array}  \right] \left[
\begin{array}{c|c}
N(z) & Y(z)\\
\hline
0 &  I_n
\end{array}  \right].$$

Now taking conformal partitions of $M(z), N(z),X(z)$ and $Y(z),$ we have

\be \label{fseq}\begin{array}{l}  \left[\begin{array}{cc|c} M_{11}(z) & M_{12}(z) & 0 \\  M_{21}(z) & M_{22}(z) & 0 \\ \hline X_1(z) & X_2(z) & I_m \end{array}\right] \left[\begin{array}{cc|c} I_{p-r} & 0 & 0 \\  0 & A_1(z) & B_1(z)  \\ \hline 0 & -C_1(z) & D_1(z) \end{array}\right] = \\[1.0cm]   \left[\begin{array}{cc|c} I_{p-\ell} & 0 & 0 \\  0 & A_2(z) & B_2(z)  \\ \hline 0 & -C_2(z) & D_2(z) \end{array}\right] \left[\begin{array}{cc|c} N_{11}(z) & N_{12}(z) & Y_1(z) \\  N_{21}(z) & N_{22}(z) & Y_2(z) \\ \hline 0 & 0 & I_n \end{array}\right], \end{array}\ee
where $ M_{11} \in \mathbb{H}(\Omega)^{(p-\ell)\times (p-r)}, M_{22} \in \mathbb{H}(\Omega)^{\ell\times r}$ and  $ N_{11} \in \mathbb{H}(\Omega)^{(p-\ell)\times (p-r)}, N_{22} \in \mathbb{H}(\Omega)^{\ell\times r}.$ Now equating the blocks, we have $$\left[
\begin{array}{c|c}
	M_{22}(z) 	& 0_{\ell\times m} \\
	\hline
	X_2(z) &   I_m
\end{array}  \right] \left[ \begin{array}{c|c}
	A_1(z) & B_1(z) \\
	\hline
	-C_1(z) &   D_1(z)
\end{array}  \right] 
= \left[
\begin{array}{c|c}
	A_2(z) & B_2(z) \\
	\hline
	-C_2(z) &   D_2(z)
\end{array}  \right] \left[
\begin{array}{c|c}
	N_{22}(z) & Y_2(z)\\
	\hline
	0_{n\times r} &   I_n
\end{array}  \right],$$  $M_{12}(z)A_1(z) = N_{12}(z) $ and $ M_{21}(z) = A_2(z) N_{21}(z).$  Thus we have \beano M(z) &=& \bmatrix{ M_{11}(z) & M_{12}(z) \\ M_{21}(z) &M_{22}(z)} = \bmatrix{ M_{11}(z) & M_{12}(z) \\ A_2(z) N_{21}(z) &M_{22}(z)}, \\ 
N(z) &=& \bmatrix{ N_{11}(z) & N_{12}(z) \\ N_{21}(z) &N_{22}(z)} = \bmatrix{ N_{11}(z) & M_{12}(z)A_1(z) \\ N_{21}(z) &N_{22}(z)}.\eeano Since $M(z)$ and $N(z)$ are unit elements, it follows that $A_2(z)$ and $M_{22}(z)$ are left coprime, and $ N_{22}(z)$ and $A_1(z)$ are right coprime, Indeed, let $$ M(z)^{-1} = \bmatrix{ X_{11}(z) & X_{12}(z) \\ X_{21}(z) &X_{22}(z)} \text{ and }  N(z)^{-1} = \bmatrix{ Y_{11}(z) & Y_{12}(z) \\ Y_{21}(z) &Y_{22}(z)}$$ be conformal partitions. Then 
\beano M(z)M(z)^{-1} & = & I_p \Longrightarrow A_2(z) N_{21}(z) X_{12}(z) + M_{22}(z)X_{22}(z) = I_\ell  \\ N(z)^{-1}N(z) &=& I_p \Longrightarrow Y_{21}(z)M_{12}(z)A_1(z) + Y_{22}(z) N_{22}(z) = I_r.\eeano Hence by Theorem~\ref{cprime2}, $A_2(z)$ and $M_{22}(z)$ are left coprime, and $A_1(z)$ and $N_{22}(z)$ are right coprime. This proves that $ \H_1(z) \sim_{fse} \H_2(z).$  
\end{proof}

\vone  Let $ \mathbf{P}(z)$ be a PMD as in (\ref{pmd}) with state matrix $ A(z)$ and rational transfer function $G(z).$ Then $\deg \det(A(z))$ is called the {\em order} of the PMD $\mathbf{P}(z)$~\cite[p.56]{vardulakis}. Further, $\mathbf{P}(z)$ is said to be 
a PMD of {\em least order} if $\mathbf{P}(z)$ is irreducible~\cite[p.81]{vardulakis}.  Theorem~\ref{equivlnt} shows that the order of a PMD is invariant under Rosenbrock's system equivalence. 
It can be  shown that if $\mathbf{S}(z)$ is a PMD with state matrix $ \widehat{A}(z) \in \C[z]^{\ell\times \ell}$ and the same transfer function $G(z)$ then $ \deg \det(A(z)) \leq \deg \det(\widehat{A}(z))$ (see, Theorem~\ref{lamd}).
 
We define least order of an AMD as follows.  Recall from Definition~\ref{ord} that if $ A \in \mathbb{H}(\Omega)^{n\times n}$ is regular then $\partial_A$ is the order of $A$, that is, $\partial_A$ is the principal divisor of $\det(A(z))$ on $\Omega.$    

\von

\begin{definition} Let  $\mathbf{H} \in \mathbb{H}(\Omega)^{(r+m) \times (r+n)}$ be  an AMD with state matrix $A \in \mathbb{H}(\Omega)^{r\times r}$ and the same transfer function $M \in \mathbb{M}(\Omega)^{m\times n}.$ Then the order of the AMD $\mathbf{H}(z)$ is defined to be the order $\partial_A$ of the state matrix $A(z).$  Further,  $\mathbf{H}(z)$ is said to be an AMD of  least order if $\widehat {\mathbf{H}} \in \mathbb{H}(\Omega)^{(\ell+m) \times (\ell+n)}$ is an AMD with state matrix $ \widehat{A} \in \mathbb{H}(\Omega)^{\ell\times \ell}$ and the transfer function $M(z)$ then $ \partial_A \leq \partial_{\widehat{A}}.$  
\end{definition} 

\von We show that the least order AMD,  irreducible AMD and AMD without decoupling zeros are equivalent concepts. 

Let $ A, B \in \mathbb{H}(\Omega)^{n\times n}$ be regular. Then we have already seen that  $$ A(z) \sim_{\Omega} B(z) \Longrightarrow \partial_A= \partial_B.$$  In particular, if $A, B \in \C[z]^{n\times n}$ then the unit elements in $\C[z]^{n\times n}$ are unimodular matrix polynomials and in such a case $A(z) \sim_{\C} B(z)$ is the unimodular equivalence. Consequently, $ \partial_A = \partial_B \Longrightarrow \deg \det(A(z)) = \deg \det(B(z)).$

 \vone 

\begin{theorem} \label{cor:fse}  Let $\H_1(z)$ and $\H_2(z)$ be AMDs  given by 
	$$ {\small \H_1(z):= \left[
	\begin{array}{c|c}
		A_1(z) & B_1(z) \\
		\hline
		-C_1(z) &   D_1(z)
	\end{array}  \right]_{(r+m) \times (r+n)} \emph{and} \  \  \H_2(z):= \left[
	\begin{array}{c|c}
		A_2(z) & B_2(z) \\
		\hline
		-C_2(z) &   D_2(z)
	\end{array}  \right]_{(\ell+m) \times (\ell+n)}}$$
	with state matrices $A_1 \in \mathbb{H}(\Omega)^{r \times r} \emph{and} \ A_2 \in \mathbb{H}(\Omega)^{\ell \times \ell}$. If $ \H_1(z) \sim_{fse} \H_2(z)$  then the following hold:
	\begin{itemize}  \item[(a)] $\partial_{A_1} = \partial_{A_2}.$ 
		\item[(b)] $ D_1(z)+ C_1(z)A_1(z)^{-1}B_1(z) = D_2(z)+ C_2(z)A_2(z)^{-1}B_2(z).$
		\item[(c)] The non-unit invariant functions in the Smith forms  of $A_1(z)$ and $A_2(z)$ are the same. 
		\item[(d)] The non-unit invariant functions in the Smith forms  of $$\bmatrix{A_1(z) & B_1(z)} \text{ and  } \bmatrix{A_2(z) & B_2(z)}$$ are the same. 
		\item[(e)]  The non-unit invariant functions in the Smith forms  of $ \bmatrix{ A_1(z)\\ C_1(z)} $ and $\bmatrix{ A_2(z)\\ C_2(z)} $ are the same.
		\item[(f)]  The non-unit invariant functions in the Smith forms  of $\H_1(z)$ and $\H_2(z)$ are the same.
	\end{itemize} 
\end{theorem} 

\vone 
\begin{proof} By Theorem~\ref{sse:fse} we have $ \H_1(z) \sim_{fse} \H_2(z) \Longleftrightarrow \H_1(z)\sim_{rse}\H_2(z).$ Hence equating  (1,1) blocks on both sides in~(\ref{fseq}), we have $$ M(z) (I_{p-r}\oplus A_1(z)) = (I_{p-\ell}\oplus A_2(z)) N(z) \Longrightarrow \det(M(z)) \det(A_1(z)) = \det(N(z)) \det(A_2(z)).$$ Since $\det(M(z))$ and $\det(N(z))$ are unit elements of $\mathbb{H}(\Omega),$ we have $ \partial_{A_1} = \partial_{A_2}.$
  
  \von
  Similarly, the remaining results follow from~(\ref{fseq}) by equating  appropriate blocks and their Smith forms.   
\end{proof}

\vone Let $ M \in \mathbb{M}(\Omega)^{m\times n}.$ Then we have seen that $M(z)$ admits a right (resp., left) MFD of the form  $ M(z) = N_R(z) D_R(z)^{-1}$ (resp., $ M(z) = D_L(z)^{-1}N_L(z) $), where $ N_L, N_R\in \mathbb{H}(\Omega)^{m\times n}$ and,  $ D_L \in \mathbb{H}(\Omega)^{m\times m}$  and $ D_R \in \mathbb{H}(\Omega)^{n\times n}$ are regular. It follows that $M(z)$ is the transfer function of the AMDs given by 
$$ \H_R(z) :=\bmatrix{ D_R(z) & I_n \\ -N_R(z) & 0_{m\times n}}_{(n+m)\times (n+n)} \text{ and }\;  \H_L(z) :=\bmatrix{ D_L(z) & N_L(z)\\ -I_m &  0_{m\times n}}_{(m+m)\times (m+n)}$$ with state matrices $D_R(z)$ and $D_L(z).$ An AMD of the form $\H_R(z)$ (or a block permutation similarity transformation of $\H_R(z)$) is called an AMD in {\em right matrix-fraction form} which we refer to as {\em  RMF-system matrix}.  Similarly, an AMD of the form $\H_L(z)$ (or a block permutation similarity transformation of $\H_L(z)$) is called an AMD in {\em left matrix-fraction form} which we refer to as {\em LMF-system matrix}. 

\vone

\begin{theorem}\label{rmf:eqv} Let $ \H_1, \H_2 \in \mathbb{H}(\Omega)^{(n+m)\times (n+n)}$ be RMF-system matrices given by 
$$\H_1(z) :=\bmatrix{ D_1(z) & I_n \\ -N_1(z) & 0_{m\times n}} \text{ and } \H_2(z) :=\bmatrix{ D_2(z) & I_n \\ -N_2(z) & 0_{m\times n}}.$$ Then $\H_1(z) \sim_{fse} \H_2(z) \Longleftrightarrow$ there exist $T_R \in \mathrm{GL}_n(\mathbb{H}(\Omega))$ such that 
$$ \bmatrix{ D_1(z) & I_n \\ -N_1(z) & 0_{m\times n}} = \bmatrix{ D_2(z) & I_n \\ -N_2(z) & 0_{m\times n}} \bmatrix{ T_R(z) & 0 \\ 0 & I_n}.$$
\end{theorem}
\vone
\begin{proof} %The proof is similar to that of \cite[Theorem~2.35]{vardulakis} except that we have to use principal divisor $\partial_{D_1}$ in place of $ \deg \det(D_1(z)).$ For completeness, we provide a proof. 
 We only need to prove the implication ($\Rightarrow$). 	Suppose that $\H_1(z) \sim_{fse} \H_2(z).$ Then there exist $M,  N, Y \in \mathbb{H}(\Omega)^{n\times n}$ and $ X \in \mathbb{H}(\Omega)^{m\times n}$ such that
\be \label{fse:eqn}\bmatrix{ M(z) & 0_{n\times m} \\ X(z) & I_m} \bmatrix{ D_1(z) & I_n \\ -N_1(z) & 0_{m\times n}} = \bmatrix{ D_2(z) & I_n \\ -N_2(z) & 0_{m\times n}} \bmatrix{ N(z) & Y(z)\\ 0_{n\times n} & I_n} \ee and $M(z), D_2(z)$ are left coprime and $ D_1(z), N(z)$ are right coprime. Multiplying both sides of  (\ref{fse:eqn}) from the right by $\bmatrix{I_n & 0_{n\times n} \\  - D_1(z) & I_n}$ we have 
\be \label{fse:eqn2}\bmatrix{ M(z) & 0_{n\times m} \\ X(z) & I_m} \bmatrix{ 0_{n\times n} & I_n \\ -N_1(z) & 0_{m\times n}} = \bmatrix{ D_2(z) & I_n \\ -N_2(z) & 0_{m\times n}} \bmatrix{ T_R(z) & Y(z)\\ -D_1(z) & I_n}, \ee
where $ T_R(z) := N(z)- Y(z) D_1(z).$ Now, equating (1, 1) and (2, 1) blocks on both sides of (\ref{fse:eqn2}), we have $D_1(z) = D_2(z)T_R(z)$ and  $ N_1(z) = N_2(z) T_R(z).$ This shows that $$ \bmatrix{ D_1(z) & I_n \\ -N_1(z) & 0_{m\times n}} = \bmatrix{ D_2(z) & I_n \\ -N_2(z) & 0_{m\times n}} \bmatrix{ T_R(z) & 0 \\ 0 & I_n}.$$ It remains to show that $ T_R \in \mathrm{GL}_n(\mathbb{H}(\Omega)).$  Since $D_1(z) = D_2(z) T_R(z),$ it follows that $ T_R(z)$ is regular and $\partial_{D_1} = \partial_{D_2}+\partial_{T_R}$, where $\partial_{D_1}, \partial_{D_2}$ and $\partial_{T_R}$, respectively,  are the principal divisors of $D_1(z), D_2(z)$ and $T_R(z)$ on $\Omega.$  By Corollary~\ref{cor:fse},  $ \partial_{D_1} = \partial_{D_2}$. Hence we have   $ \partial_{T_R} = 0.$  In other words, $\det(T_R(z))$ is a unit element of $\mathbb{H}(\Omega)$ which in turn shows that $ T_R \in \mathrm{GL}_n(\mathbb{H}(\Omega)).$ 
\end{proof}

\vone 
By similar arguments, we have the following result. 
\von  
\begin{corollary} Let $ \H_1, \H_2 \in \mathbb{H}(\Omega)^{(m+m)\times (m+n)}$ be LMF-system matrices given by 
	$$\H_1(z) :=\bmatrix{ D_1(z) & N_1(z) \\ -I_m & 0_{m\times n}} \text{ and } \H_2(z) :=\bmatrix{ D_2(z) & N_2(z) \\ -I_m & 0_{m\times n}}.$$ Then $\H_1(z) \sim_{fse} \H_2(z) \Longleftrightarrow$ there exist $T_L \in \mathrm{GL}_m(\mathbb{H}(\Omega))$ such that 
	$$ \bmatrix{ D_1(z) & N_1(z) \\ -I_m & 0_{m\times n}} = \bmatrix{ T_L(z) & 0 \\ 0 & I_m} \bmatrix{ D_2(z) & N_2(z) \\ -I_m & 0_{m\times n}} .$$
	
\end{corollary}

\vone 
We now show that an AMD, under appropriate assumption, is Fuhrmann system equivalent to an RMF-system matrix.  \von

%Let $(C, A, B) \in \mathbb{H}(\Omega)^{m\times r} \times \mathbb{H}(\Omega)^{r\times r} \times \mathbb{H}(\Omega)^{r\times n} $ with $A(z)$ being regular and $D \in \mathbb{H}(\Omega)^{m\times n}.$

\vone \begin{theorem} \label{th:eqv1}  Consider the AMD given by  $$\H_1(z) := \bmatrix{ A(z) & B(z) \\ -C(z) & D(z)} \in \mathbb{H}(\Omega)^{(r+m)\times (r+n)} $$ with state matrix $A \in \mathbb{H}(\Omega)^{r\times r}.$  If $ A(z)$ and $B(z)$ are left coprime then there exists an RMF-system matrix  $$  \H_2(z) := \bmatrix{ D_R(z) & I_n \\ -N_R(z) & 0_{m\times n}} \in \mathbb{H}(\Omega)^{(n+m)\times (n+n)}$$ such that 
$\H_1(z) \sim_{fse} \H_2(z).$ Further, if $ A(z)$ and $ C(z)$ are right coprime then $ N_R(z)$ and $D_R(z)$ are right coprime.

\end{theorem}

\begin{proof}  Since $A(z)$ and $B(z)$ are left coprime, by an analogue of Theorem~\ref{cprime2}(e), there exists $N \in  \mathbb{H}(\Omega)^{n\times r}$ and $ Y \in \mathbb{H}(\Omega)^{n\times n}$ such that $$T(z) := \bmatrix{ A(z) & B(z) \\ -N(z) & -Y(z)} \in \mathrm{GL}_{r+n}(\mathbb{H}(\Omega)).$$ Consider the conformal partition of  $T(z)^{-1}$  given by $$ T(z)^{-1} = \bmatrix{T_1(z) & T_2(z) \\ M(z) & D_R(z)} \in \mathrm{GL}_{r+n}(\mathbb{H}(\Omega)),$$ where  $\text{ with } T_1 \in \mathbb{H}(\Omega)^{r\times r}$ and $  D_R \in \mathbb{H}(\Omega)^{n\times n}.$  Then 
	\be\label{eq3} \bmatrix{T_1(z) & T_2(z) \\ M(z) & D_R(z)}  \bmatrix{ A(z) & B(z) \\ -N(z) & -Y(z)} = T(z)^{-1}T(z) =\bmatrix{ I_r & 0 \\ 0 & I_n}  \ee gives $  M(z)A(z) = D_R(z) N(z) $ and $ M(z) B(z) = D_R(z) Y(z) + I_n$ which can be written together as 
	\be \label{eq1}  M(z) \bmatrix{ A(z) & B(z)} = \bmatrix{ D_R(z) & I_n} \bmatrix{ N(z) & Y(z) \\ 0_{n\times r} & I_n}. \ee Again using the fact that $ T(z) T(z)^{-1} = I_{r+n},$ we have 
$$\bmatrix{ A(z) & B(z) \\ -C(z) & D(z)} \bmatrix{T_1(z) & T_2(z) \\ M(z) & D_R(z)} = \bmatrix{ I_r & 0_{r\times n} \\ - X(z) & N_R(z)} =  \bmatrix{ I_r & 0_{r\times m} \\ - X(z) & I_m}  \bmatrix{ I_r & 0 \\ 0 & N_R(z)},$$ where $ X(z) := C(z) T_1(z) - D(z)M(z)$ and $N_R(z) := D(z)D_R(z) -C(z)T_2(z).$ Hence 
\beano \bmatrix{ I_r & 0_{r\times m} \\ X(z) & I_m}  \bmatrix{ A(z) & B(z) \\ -C(z) & D(z)}  & = &   \bmatrix{ I_r & 0 \\ 0 & N_R(z)} \bmatrix{T_1(z) & T_2(z) \\ M(z) & D_R(z)}^{-1} \\ & = & \bmatrix{ I_r & 0 \\ 0 & N_R(z)}\bmatrix{ A(z) & B(z) \\ -N(z) & -Y(z)}\eeano yields (by equating the last block row) \be \label{eq2} X(z) \bmatrix{A(z) & B(z)} + \bmatrix{-C(z) & D(z)} = -N_R(z) \bmatrix{ N(z) & Y(z)}. \ee Now writing (\ref{eq1}) and (\ref{eq2}) together we have 
$$ \bmatrix{ M(z) & 0_{n\times m} \\ X(z) & I_m}  \bmatrix{ A(z) & B(z) \\ -C(z) & D(z)} = \bmatrix{ D_R(z) & I_n \\ -N_R(z) & 0_{m\times n}} \bmatrix{ N(z) & Y(z) \\ 0_{n\times r} & I_n}.$$
It follows from (\ref{eq3}) that $ A(z)$ and $N(z)$ are right coprime and  $ M(z)$ and $D_R(z)$ are left coprime. It remains to show that $D_R(z)$ is regular. 
Let $ f \in \mathbb{H}(\Omega)^n$ and $D_Rf = 0.$ Now $T(z) T(z)^{-1} = I_{r}\oplus I_n$ gives $A T_2 f + B D_R f = 0 \Longrightarrow AT_2f = 0 \Longrightarrow  T_2f = 0$ since $A(z)$ is regular. This shows that $\bmatrix{ T_2 \\ D_R} f = 0 \Longrightarrow f =0$ since $T(z)^{-1}$ is a unit element. Hence $D_R(z)$ is regular. 
This proves that $ \H_1(z) \sim_{fse} \H_2(z).$ 

\vone 
By Theorem~\ref{cor:fse}, the matrices $ \bmatrix{A(z) \\ -C(z)} $  and $\bmatrix{ D_R(z) \\ -N_R(z)}$ have the identical non-unit invariant functions in their Smith forms. Thus if $ A(z)$ and $ C(z)$ are right coprime then $\bmatrix{ I_r \\ 0_{m\times r}}$ is the Smith form of $ \bmatrix{A(z) \\ -C(z)}.$ Consequently, $\bmatrix{ I_n \\ 0_{m\times n}}$ is the Smith form of  $\bmatrix{ D_R(z) \\ -N_R(z)}$ which in turn shows that $D_R(z)$ and $N_R(z)$ are right coprime. 
 \end{proof}

\vone

We mention that our proof of Theorem~\ref{th:eqv1} is concise and simpler than the proof given for PMDs in \cite[Theorem~2.40]{vardulakis}.

\vone 
We have the following result whose proof is similar to that of Theorem~\ref{th:eqv1}.  

%Let $(C, A, B) \in \mathbb{H}(\Omega)^{m\times r} \times \mathbb{H}(\Omega)^{r\times r} \times \mathbb{H}(\Omega)^{r\times n} $ with $A(z)$ being regular and $D \in \mathbb{H}(\Omega)^{m\times n}.$
\vone 
\begin{theorem} \label{cor:eqv2}   Consider an AMD given by $$\H_1(z) := \bmatrix{ A(z) & B(z) \\ -C(z) & D(z)} \in \mathbb{H}(\Omega)^{(r+m)\times (r+n)}$$ with state matrix $A \in \mathbb{H}(\Omega)^{r\times r}.$ If $ A(z)$ and $C(z)$ are right coprime then there exists an LMF-system matrix  $$  \H_2(z) := \bmatrix{ D_L(z) & N_L(z) \\ -I_m & 0_{m\times n}} \in \mathbb{H}(\Omega)^{(m+m)\times (m+n)}$$ such that 
	$\H_1(z) \sim_{fse} \H_2(z).$ Further, if $ A(z)$ and $ B(z)$ are left coprime then  $D_L(z)$ and $ N_L(z)$are left coprime.
\end{theorem}

\vone
We are now ready to prove that two irreducible AMDs are Rosenbrock system equivalent (Fuhrmann system equivalent) if and only if they have the same transfer function. \vone

\begin{theorem} \label{eqvsys} Let $\H_1 \in \mathbb{H}(\Omega)^{(r+m)\times(r+n)}$ and $\H_2 \in \mathbb{H}(\Omega)^{(\ell+m)\times \ell+n)}$ be irreducible AMDs   given by 
	$$ \H_1(z):= \left[
	\begin{array}{c|c}
		A_1(z) & B_1(z) \\
		\hline
		-C_1(z) &   D_1(z)
	\end{array}  \right]_{(r+m) \times (r+n)} \emph{and} \  \  \H_2(z):= \left[
	\begin{array}{c|c}
		A_2(z) & B_2(z) \\
		\hline
		-C_2(z) &   D_2(z)
	\end{array}  \right]_{(\ell+m) \times (\ell+n)}$$ with transfer functions $M_i(z) := D_i(z)+C_i(z)A_i(z)^{-1}B_i(z), \; i= 1,2,$ 
	where $A_1 \in \mathbb{H}(\Omega)^{r \times r} \emph{and} \ A_2 \in \mathbb{H}(\Omega)^{\ell \times \ell}$ are  state matrices. Then $$ \H_1(z) \sim_{rse} \H_2(z) \Longleftrightarrow \H_1(z) \sim_{fse} \H_2(z) \Longleftrightarrow M_1(z) = M_2(z).$$
	
\end{theorem}

\begin{proof} We have already seen that $ \H_1(z) \sim_{rse} \H_2(z) \Longleftrightarrow \H_1(z) \sim_{fse} \H_2(z)$ and in such a case, by Theorem~\ref{th:sse} and  Theorem~\ref{cor:fse}(b), we have $M_1(z) = M_2(z).$ 
	
	\vone Conversely, if $ M_1(z) = M_2(z)$ then we now show that  $ \H_1(z) \sim_{fse} \H_2(z)$. By Theorem~\ref{th:eqv1}, there exist RMF-system matrices  $$  \mathbf{S}_i(z) := \bmatrix{ D_i(z) & I_n \\ -N_i(z) & 0_{m\times n}} \in \mathbb{H}(\Omega)^{(n+m)\times (n+n)}$$ with transfer function  $ N_i(z)D_i(z)^{-1}$  such that $ \H_i(z) \sim_{fse} \mathbf{S}_i(z)$ for $ i = 1,2.$ Hence by Theorem~\ref{cor:fse}(b), $M_i(z) = N_i(z)D_i(z)^{-1}$ for $ i=1,2.$  Now $$M_1(z) = M_2(z)  \Longrightarrow N_1(z)D_1(z)^{-1} = N_2(z)D_2(z)^{-1}.$$  Since $\H_i(z)$ is irreducible,  by Theorem~\ref{th:eqv1} the MFD $N_i(z)D_i(z)^{-1}$ is right coprime for $ i=1,2.$  Hence by Theorem~\ref{eqv:mfd}, there exist $ T_R \in \mathrm{GL}_n(\mathbb{H}(\Omega))$ such that $$ N_1(z) = N_2(z) T_R(z) \text{ and } D_1(z) = D_2(z)T_R(z).$$ This shows that $$ \bmatrix{ D_1(z) & I_n \\ -N_1(z) & 0_{m\times n}} = \bmatrix{ D_2(z) & I_n \\ -N_2(z) & 0_{m\times n}} \bmatrix{ T_R(z) & 0 \\ 0 & I_n}.$$ Now, by Theorem~\ref{rmf:eqv}, we have $ \mathbf{S}_1(z) \sim_{fse} \mathbf{S}_2(z).$ Consequently, we have  $$ \H_1(z) \sim_{fse} \mathbf{S}_1(z)\sim_{fse} \mathbf{S}_2(z)\sim_{fse}\H_2(z),$$ that is, $ \H_1(z) \sim_{fse} \H_2(z).$ This completes the proof. 
\end{proof}

\vone We now examine the conditions under which an LMF-system matrix is Fuhrmann system equivalent to an RMF-system matrix. \vone 

\begin{theorem} Consider the LMF-system matrix $\H_L(z) \in \mathbb{H}(\Omega)^{(m+m)\times (m+n)}$ and the RMF-system matrix $\H_R(z) \in \mathbb{H}(\Omega)^{(n+m)\times (n+n)}$ given by 
$$ \H_L(z) := \bmatrix{ D_L(z) & N_L \\ - I_m & 0_{m\times n}} \text{ and } \;  \H_R(z) := \bmatrix{ D_R(z) & I_n \\ -N_R(z) & 0_{m\times n}}.$$	Then $$ \H_L(z) \sim_{fse}\H_R(z) \Longleftrightarrow \left\{ \begin{array}{l} (a) \; D_L(z)^{-1} N_L(z)  = N_R(z)D_R(z)^{-1} \\ (b) \; D_L(z), N_L(z) \text{ are left coprime} \\ (c) \; D_R(z), N_R(z) \text{ are right coprime.} \end{array}  \right. $$

\end{theorem} 

\vone \begin{proof}  First, note that $ D_L(z)^{-1}N_L(z)$ is the transfer function of $\H_L(z)$ and $ N_R(z)D_L(z)^{-1}$ is the transfer function of $\H_R(z).$ Thus, if the conditions (a), (b) and (c)  hold then $\H_L(z)$ and $\H_R(z)$ are irreducible and have the same transfer function. Hence by Theorem~\ref{eqvsys}, we have  $ \H_L(z) \sim_{fse} \H_R(z).$ \vone 

Conversely, suppose that  $ \H_L(z) \sim_{fse} \H_R(z).$ Then by Theorem~\ref{cor:fse}, the non-unit invariant functions in the Smith forms of $\bmatrix{D_L(z) & N_L(z)} $ and $ \bmatrix{ D_R(z) & I_n}$ are identical. Also, the   non-unit invariant functions in the Smith forms of $ \bmatrix{ D_L \\ -I_m}$ and $\bmatrix{ D_R \\ -N_R(z)} $ are identical.  Since $ D_R(z) $ and $I_n$ are left coprime, the invariant functions in the Smith form  of $ \bmatrix{ D_R(z) & I_n}$ are all unit elements. This implies that the invariant functions in the Smith form of $\bmatrix{D_L(z) & N_L(z)} $ are all unit elements. Hence $D_L(z) $ and $N_L(z)$ are left coprime. Similarly $ D_L(z)$ and $I_m$ are right coprime which implies that $ D_R(z) $ and $ N_R(z)$ are right coprime. Again, by Theorem~\ref{cor:fse}, we have $ D_L^{-1}N_L(z) = N_R(z) D_R(z)^{-1}.$  
 \end{proof} 
 
 \vone 
 We now examine relationships between canonical forms of AMD, state matrix and the transfer function. \vone 
 
 % Let $(C, A, B) \in \mathbb{H}(\Omega)^{m\times r} \times \mathbb{H}(\Omega)^{r\times r} \times \mathbb{H}(\Omega)^{r\times n} $ with $A(z)$ being regular and $D \in \mathbb{H}(\Omega)^{m\times n}.$
 \begin{theorem} \label{th:main}
 	 Consider an irreducible  AMD given by  $$\H(z) := \bmatrix{ A(z) & B(z) \\ -C(z) & D(z)} \in \mathbb{H}(\Omega)^{(r+m)\times (r+n)} $$ with transfer function $M(z) := D(z) + C(z)A(z)^{-1}B(z) \in \mathbb{H}(\Omega)^{m\times n}$ and state matrix $A \in \mathbb{H}(\Omega)^{r\times r}$. Suppose that $\nrank(M) = p.$ Let $$ \Sigma_{M}(z) := \phi_1(z)/\psi_1(z) \oplus \cdots \oplus \phi_p(z)/\psi_p(z) \oplus 0_{m-p, n-p}$$ be the Smith-McMillan form of $M(z).$  Then 
 	 \beano S_A(z) &:=& I_{r-p} \oplus \diag(\psi_p(z), \psi_{p-1}(z), \psi_1(z)) \\ S_{\H}(z) &:= & I_r \oplus \diag(\phi_1(z), \ldots, \phi_p(z)) \oplus 0_{m-p, n-p} \eeano are the Smith forms of $A(z)$ and $\H(z),$ respectively.
 	 	 In particular, we have 	\begin{itemize} \item[(a)] $\sig_{\Omega}(M) = \sig_{\Omega}(\H) $ and $ \mathrm{Ind}_e(\lam, M) =\mathrm{Ind}_e(\lam, \H) $ for all $\lam \in \sig_{\Omega}(M).$  \item[(b)]  $ \wp_{\Omega}(M) = \sig_{\Omega}(A)$ and $ \mathrm{Ind}_p(\lam, M) = \mathrm{Ind}_e(\lam, A)$ for all $ \lam \in \wp_{\Omega}(M).$ \end{itemize} 
 \end{theorem}
 
 \begin{proof} Since $\H(z)$ is irreducible, by Theorem~\ref{th:eqv1} there exists an irreducible RMF-system matrix (with $D_R(z)$ and $ N_R(z)$ are right coprime) $$ \mathbf{S}(z) :=  \bmatrix{ D_R(z) & I_n \\ -N_R(z) & 0_{m\times n}} \in \mathbb{H}(\Omega)^{(n+m)\times (n+n)}$$ such that 
 	$ \H(z) \sim_{fse} \mathbf{S}(z).$ By Theorem~\ref{cor:fse},  $ M(z) = N_R(z) D_R(z)^{-1}$ is a right coprime MFD. Now by Theorem~\ref{rcmfd}, $ S_{D_R}(z) :=  I_{n-p} \oplus \diag(\psi_p(z), \psi_{p-1}(z), \psi_1(z))$ is the Smith form of $D_R(z)$ and $ S_{N_R}(z) := \diag(\phi_1(z), \ldots, \phi_p(z)) \oplus 0_{m-p, n-p} $ is the Smith form of $N_R(z).$ 
 	
 	\vone 
 	
 	By Theorem~\ref{cor:fse}, the non-unit invariant functions in the Smith forms of $A(z)$ and $D_R(z)$ are the same. Hence  $S_A(z) = I_{r-p} \oplus \diag(\psi_p(z), \psi_{p-1}(z), \psi_1(z))$ is the Smith form of $A(z).$  Next, we have 
 	$$  \bmatrix{ D_R(z) & I_n \\ -N_R(z) & 0_{m\times n}}  \bmatrix{ 0_{n\times n} & -I_n \\ I_n & D_R(z)} = \bmatrix{ I_n & 0_{n\times n} \\ 0_{m\times n} & N_R(z)} = : \H_1(z),$$ where the second matrix on the left hand side is a unit element. This shows that $\mathbf{S}(z)$ and $\H_1(z)$ have the same Smith form and $ S_{\H_1}(z) = I_n \oplus  \diag(\phi_1(z), \ldots, \phi_p(z)) \oplus 0_{m-p, n-p}$ is the Smith form of $ \H_1(z).$ By Theorem~\ref{cor:fse}, the non-unit invariant functions in the Smith forms of $\H(z)$ and $\mathbf{S}(z)$ are the same. Consequently, $$S_{\H}(z) = I_r \oplus  \diag(\phi_1(z), \ldots, \phi_p(z)) \oplus 0_{m-p, n-p}$$ is the Smith form of $\H(z).$ The proofs of the remaining results are immediate.   	
 \end{proof}

\vone 
Finally, we have the following result that establishes equivalence between an irreducible AMD  and  an irreducible RMF system matrix. Obviously, similar result holds for an irreducible LMF-system matrix. \vone 

\begin{theorem} \label{th:main2} Let $M \in \mathbb{M}(\Omega)^{m\times n}.$ Let $ M(z) = N_R(z)D_R(z)^{-1}$ be a right coprime MFD, where $D_R(z)$ is regular and $ (N_R, D_R) \in \mathbb{H}(\Omega)^{m\times n} \times \mathbb{H}(\Omega)^{n\times n}$. Let 
	$$ M(z) = D(z)+C(z)A(z)^{-1}B(z)$$ be a minimal realization,  where $(C, A, B, D) \in \mathbb{H}(\Omega)^{m\times r} \times \mathbb{H}(\Omega)^{r\times r} \times \mathbb{H}(\Omega)^{r\times n} \times \mathbb{H}(\Omega)^{m\times n}$ and   $A(z)$ is regular. Consider the  AMD $\H(z)$ and the RMF-system matrix $\mathbf{S}(z)$ given by  	
	$$\H(z) := \bmatrix{ A(z) & B(z) \\ -C(z) & D(z)} \text{ and } \mathbf{S}(z) :=  \bmatrix{ D_R(z) & I_n \\ -N_R(z) & 0_{m\times n}}. $$ 	Let $ p \geq \max(r, n).$  Then the following hold: 
	\begin{itemize}
		\item[(a)] $ \H(z)\sim_{fse} \mathbf{S}(z)$ or equivalently $ \H(z)\sim_{rse} \mathbf{S}(z)$. 
		\item[(b)] $I_{p-r} \oplus \H(z) \sim_{rse} I_{p-n}\oplus \mathbf{S}(z) \sim_{\Omega} I_{p} \oplus N_R(z).$
		\item[(c)]  $ I_{p-r}\oplus A(z) \sim_{\Omega} I_{p-n} \oplus D_R(z)$. 
	\end{itemize}	
 Let $ S_{\H}(z), S_{\mathbf{S}}(z)$ and $S_{N_R}(z)$ be the Smith forms of $\H(z), \mathbf{S}(z)$ and $ \N_R(z),$ respectively. Then we have  $ I_{p-r}\oplus S_A(z) = I_{p-n} \oplus S_{D_R}(z) $ and $$ I_{p-r} \oplus S_{\H}(z) = I_{p-n}\oplus S_{\mathbf{S}}(z) = I_{p} \oplus S_{N_R}(z).$$

\noindent In particular, $\sig_{\Omega}(M) = \sig_{\Omega}(\H) = \sig_{\Omega}(\mathbf{S})= \sig_{\Omega}(N_R) $ and $ \wp_{\Omega}(M) = \sig_{\Omega}(A) = \sig_{\Omega}(D_R).$ Further, $ \mathrm{Ind}_e(\lam, M) = \mathrm{Ind}_e(\lam, \H)  = \mathrm{Ind}_e(\lam, \mathbf{S}) = \mathrm{Ind}_e(\lam, N_R)$ for $ \lam \in \sig_{\Omega}(M)$ and $ \mathrm{Ind}_p(\lam, M)= \mathrm{Ind}_e(\lam, A) = \mathrm{Ind}_e(\lam, D_R)$ for $ \lam \in \wp_{\Omega}(M).$
\end{theorem}

\vone \begin{proof} Since $\H(z)$ and $\mathbf{S}(z)$ are irreducible AMDs having the same transfer function, by Theorem~\ref{eqvsys}, we have $ \H(z)\sim_{fse} \mathbf{S}(z).$ By Theorem~\ref{sse:fse}, we have $ \H(z)\sim_{rse}\mathbf{S}(z).$ 
 Since	 $\H(z)\sim_{rse} \mathbf{S}(z)$, there exist  $M, N \in \mathrm{GL}_p(\mathbb{H}(\Omega)),$ $X \in \mathbb{H}(\Omega)^{m \times p}$ and $Y \in \mathbb{H}(\Omega)^{p \times n}$  such that 
	\be \label{eq:sse2}\left[
	\begin{array}{c|c}
		M(z)
		& 0 \\
		\hline
		X(z) &  I_m 
	\end{array}  \right] \left[
	\begin{array}{c|c}
		I_{p-r} & 0 \\
		\hline
		0&  \H(z)
	\end{array}  \right]\left[
	\begin{array}{c|c}
		N(z) & Y(z)\\
		\hline
		0 &  I_n
	\end{array}  \right] =  \left[
	\begin{array}{c|c}
		I_{p-n}  & 0 \\
		\hline
		0& \mathbf{S}(z)
	\end{array}  \right].\ee This shows that  $I_{p-r} \oplus \H(z) \sim_{rse} I_{p-n}\oplus \mathbf{S}(z).$  Now elementary block column operations show that  
	$$  \mathbf{S}(z) :=  \bmatrix{ D_R(z) & I_n \\ -N_R(z) & 0_{m\times n}} \sim_{\Omega} \bmatrix{ I_n  & 0 \\ 0 & N_R(z)} \Longrightarrow I_{p-n}\oplus \mathbf{S}(z) \sim_{\Omega} I_{p} \oplus N_R(z). $$ Now equating (1, 1) block on both sides of (\ref{eq:sse2}), we have $$M(z) (I_{p-r}\oplus A(z)) N(z) = I_{p-n} \oplus D_R(z) \Longrightarrow  I_{p-r}\oplus A(z) \sim_{\Omega} I_{p-n} \oplus D_R(z).$$   The proofs of remaining assertions follow immediately from Theorem~\ref{rcmfd} and Theorem~\ref{th:main}.   
	\end{proof}

\vone

 The next results characterizes least order AMDs. It shows that the least order and irreduciblity of an AMD and least order of the associated transfer function are equivalent concepts. This result may be new for PMDs as well.  \vone 

\begin{theorem}[least order AMD] \label{lamd} Consider an AMD given by 	$$\H(z) := \bmatrix{ A(z) & B(z) \\ -C(z) & D(z)} \in \mathbb{H}(\Omega)^{(r+m)\times (r+n)}$$ with transfer function $M(z) := D(z) + C(z) (A(z))^{-1} B(z).$  Then the following statements are equivalent. 
	\begin{itemize}
		\item[(a)] The AMD $\mathbf{H}(z)$ is of least order. 
		\item[(b)] The order $\partial_A$ of the AMD $\mathbf{H}(z)$ = the least order $\nu(M)$ of $M(z).$ 
		\item[(c)] The AMD is irreducible. 
	\end{itemize}
	
\end{theorem} 

\begin{proof} Suppose that (a) holds. Let $ M(z) = N_R(z)D_R(z)^{-1}$ be a right coprime MFD. Then $\nu(M) = \partial_D$ is the least order of $M(z).$ Now, consider the RMF system matrix $$  \mathbf{S}(z) :=  \bmatrix{ D_R(z) & I_n \\ -N_R(z) & 0_{m\times n}}.$$ Since $ \mathbf{S}(z)$ is an AMD with state matrix $D_R(z)$ and the transfer function $M(z)$, we have $ \partial_A \leq \partial_{D_R}$ and  $ \supp(\partial_A) \subset \supp(\partial_{D_R}).$ By Theorem~\ref{rcmfd}, we have  $ \supp(\partial_{D_R})  = \wp_{\Omega}(M) \subset \sig_{\Omega}(A) = \supp(\partial_A)$ and $\partial_{D_R} \leq \partial_{A}$ which shows that $ \supp(\partial_A) = \supp(\partial_{D_R})$ and $ \partial_A = \partial_{D_R} = \nu(M).$ This proves (b).

	Next, suppose that $\partial_A = \nu(M).$ Then $ \wp_{\Omega}(M) = \supp(\partial_A)= \sig_{\Omega}(A)$ and $ \partial_A(\lam)$ is the multiplicity of the pole $\lam$ of $M(z).$  We show that $A(z), C(z)$ are right coprime and $A(z), B(z)$ are left coprime. Let $Q(z)$ be a gcrd of $A(z)$ and $C(z).$ Then $ A(z) = A_R(z) Q(z)$ and $ C(z) = C_R(z)Q(z)$, where $ A_R(z)$ and $C_R(z)$ are right coprime.  Now $ M(z) = D(z) + C(z) (A(z))^{-1} B(z) = D(z) + C_R(z)( A_R(z))^{-1}B(z)$ shows that $\sig_{\Omega}(A) = \wp_{\Omega}(M) \subset \sig_{\Omega}(A_R)$ and $ \partial_{A_R}(\lam)  \geq \partial_A(\lam)$ for $\lam \in \wp_{\Omega}(M)$. 
	Since $ A(z) = A_R(z) Q(z)$, we have $\sig_{\Omega}(A_R) \subset \sig_{\Omega}(A)$ and $ \partial_A = \partial_{A_R} + \partial_Q$ which implies that $ \partial_A \geq \partial_{A_R}$. This shows that $ \partial_A = \partial_{A_R}$. Consequently, $\partial_Q =0$ which implies that $Q(z)$ is a unit element, that is, $Q \in \mathrm{GL}_n(\mathbb{H}(\Omega)).$  Hence $A(z)$ and $C(z)$ are right coprime.   
	A similar proof shows that $ A(z)$ and $B(z)$ are left coprime. Consequently the AMD $\mathbf{H}(z)$ is irreducible. This proves (c).
	
Finally, assume that (c) holds. By Theorem~\ref{th:main}, $ \wp_{\Omega}(M) = \sig_{\Omega}(A) = \supp(\partial_A)$ and $\partial_A(\lam)$ is the multiplicity of the pole $\lam$ of $M(z).$ Let 	$$\H_1(z) := \bmatrix{ A_1(z) & B_1(z) \\ -C_1(z) & D_1(z)} \in \mathbb{H}(\Omega)^{(\ell+m)\times (\ell+n)}$$ be an AMD with state matrix $A_1(z)$ and the  transfer function $M(z)$. Then  we have $M(z) = D_1(z) + C_1(z) (A_1(z))^{-1} B_1(z)$ which shows that  $\sig_{\Omega}(A) = \wp_{\Omega}(M) \subset \sig_{\Omega}(A_1)$. Further, if $\lam $ is a pole of $M(z)$ then the multiplicity of $\lam$ as a zero of  $ \det(A_1(z))$ is at least $\partial_A(\lam).$ This shows that $ \partial_A  \leq \partial_{A_1}$. Hence $\mathbf{H}(z)$ is an AMD of least order. This proves (a).    
\end{proof} 

\vone 

{\bf Decoupling zeros of AMD:} As in the case of PMDs, we now show that if an AMD is not irreducible then an irreducible AMD can be extracted by removing decoupling zeros.   Consider an AMD 
$$\H(z) := \bmatrix{ A(z) & B(z) \\ -C(z) & D(z)} \in \mathbb{H}(\Omega)^{(r+m)\times (r+n)}$$ with state matrix $A(z)$ and transfer function $M(z) = D(z) + C(z) (A(z))^{-1} B(z).$ Let $ Q_L(z)$ be a greatest common left divisor of $A(z) $ and $B(z)$. Then $ A(z) = Q_L(z) \widetilde{A}(z)$ and $ B(z) = Q_L(z) \widehat{B}(z)$, where $\widetilde{A}(z)$ and $\widehat{B}(z)$ are left coprime. Next, let $ Q_R(z)$ be a gcrd of $ \widetilde{A}(z)$ and $C(z).$ Then $ \widetilde{A}(z) = \widehat{A}(z) Q_R(z)$ and $ C(z) = \widehat{C}(z) Q_R(z)$, where  $ \widehat{A}(z) $ and $ \widehat{C}(z)$ are right coprime.  Note that $A(z) = Q_L(z) \widehat{A}(z) Q_R(z).$ Consider the AMD 
$$\widehat{\H}(z) := \bmatrix{ \widehat{A}(z) & \widehat{B}(z) \\ - \widehat{C}(z) & D(z)} \in \mathbb{H}(\Omega)^{(r+m)\times (r+n)}  $$ with  transfer function $ D(z) + \widehat{C}(z) (\widehat{A}(z))^{-1} \widehat{B}(z) = D(z) + C(z) (A(z))^{-1} B(z) = M(z)$ and state matrix $\widehat{A}(z)$. The AMD $\widehat{\H}(z)$ is irreducible. Indeed, since $\widehat{A}(z)$ and $\widehat{C}(z)$ are right coprime, we need to show that $\widehat{A}(z)$ and $\widehat{B}(z)$ are left coprime.  Let $D_L(z)$ be a greatest common left divisor of $ \widehat{A}(z) $ and $\widehat{B}(z)$. Then $ D_L(z)$ is a common left  divisor of $ \widetilde{A}(z)=\widehat{A}(z) Q_R(z)$ and $ \widehat{B}(z).$ Since  $ \widetilde{A}(z)$ and $ \widehat{B}(z)$ are left coprime, $D_L(z)$ is a unit element. This shows that $\widehat{A}(z)$ and $\widehat{B}(z)$ are left coprime. Hence $\widehat{\H}(z)$ is irreducible and by Theorem~\ref{lamd}  $\widehat{\H}(z)$ is an AMD of least order. Note that $\partial_A$ is the order of $\H(z)$ and $  \partial_{\widehat{A}}$ is the order of $\widehat{\H}(z).$ Since $ A(z) = Q_L(z) \widehat{A}(z) Q_R(z)$, we have $\partial_A = \partial_{Q_L} + \partial_{\widehat{A}}+ \partial_{Q_R} $ which shows that $\partial_{\widehat{A}} \leq \partial_A$ and the equality holds when $Q_L(z)$ and $Q_R(z)$ are unit elements, that is, when $\H(z)$ is irreducible.   

\vone 
Define $ \sig^{d}_{\Omega}(\H) := \sig_{\Omega}(Q_L) \cup \sig_{\Omega}(Q_R)$. Then  $\sig^{d}_{\Omega}(\H)$ is called the {\em decoupling zeros} of $\H(z).$ The spectrum $\sig_{\Omega}^{id}(\H) :=\sig_{\Omega}(Q_L)$ is called the {\em input-decoupling zeros} of $\H(z).$  Let $ D_R(z)$ be a gcrd of $A(z)$ and $C(z)$. Then $\sig_{\Omega}^{od}(\H) := \sig_{\Omega}(D_R)$ is called the {\em output-decoupling zeros} of $\H(z)$  and  $ \sig^{iod}_{\Omega}(\H) := \sig_{\Omega}(D_R)\setminus \sig_{\Omega}(Q_R)$ is called the {\em input-output decoupling zeros} of $\H(z).$ Therefore, we have
$$  \sig^{d}_{\Omega}(\H) =  \sig_{\Omega}^{id}(\H)\cup \sig_{\Omega}^{od}(\H) \setminus  \sig^{iod}_{\Omega}(\H).$$  Note that $ \sig_{\Omega}(A) = \sig_{\Omega} (Q_L) \cup \sig_{\Omega}(\widehat{A}) \cup \sig_{\Omega}(Q_R) = \wp_{\Omega}(M) \cup \sig^{d}_{\Omega}(\H)$.  The spectrum  $\sig_{\Omega}(A) = \wp_{\Omega}(M) \cup \sig^{d}_{\Omega}(\H)$ is called the {\em poles} of the AMD $\H(z)$ and  $\sig_{\Omega}(\H) \cup \sig^{d}_{\Omega}(\H) $ is called the {\em zeros}  of the AMD $\H(z).$  Observe that $ \H(z)$ is irreducible $\Longleftrightarrow  \sig^{d}_{\Omega}(\H) = \emptyset.$ 

%zeros of AMD $\H(z) $ 

\vone Finally,  consider an LTI system with a single control delay and a single state delay given by \cite{tdsbook1, tdsbook2}  
\beano \frac{d\mathbf{x}}{dt} &=& A\mathbf{x}(t) +  B \mathbf{u}(t-\tau), \; t >0  \\ \mathbf{y}(t) &=& C \mathbf{x}(t-h). \eeano
Then system matrix $\H(z)$ and the transfer function $M(z)$  are given by 
 $$ \H(z) := \left[ \begin{array}{c|c} z I_n -A  & Be^{-\tau z} \\ \hline -C e^{-h z}& 0 \end{array} \right] \text{ and } M(z) = e^{-h z}C ( z I_n -A)^{-1}B e^{-\tau z}.$$ In this case,  $M(z)$ is a transcendental meromorphic matrix with a finite number of poles in $\C$ and an isolated essential singularity at infinity. Further, we have  
$$ \H(z) := \left[ \begin{array}{c|c} z I_n -A  & Be^{-\tau z} \\ \hline -C e^{-h z} & 0 \end{array} \right] \sim_{\C} \left[ \begin{array}{c|c} z I_n -A  & B \\ \hline -C & 0 \end{array} \right] =: \mathbf{S}(z) $$  and  $ M(z) = e^{-hz }C ( z I_n -A)^{-1}B e^{-\tau z} \sim_{\C} C ( z I_n -A)^{-1}B =: G(z).$ This shows that $\H(z)$ and $\mathbf{S}(z)$ have the same Smith form and $M(z)$ and $ G(z)$ have the same Smith-McMillan form. Further, $\H(z)$ and $\mathbf{S}(z)$ have the same decoupling (resp., input-decoupling, out-put decoupling, input-output decoupling) zeros. In other words, zeros and poles of the AMD $\H(z)$ and PMD $\mathbf{S}(z)$ are the same. Also, pointwise controllability~\cite{tdsbook1, tdsbook2} (resp., observability) of the LTI time-delay system follow from controllability (resp., observability) of the pair $(A, B)$ (res., $(A, C)$).  In particular, $\H(z)$ is irreducible if and only if $\mathbf{S}(z)$ is irreducible.

\von
{\bf Conclusion:} We have presented a theoretical framework for analyzing MFDs and AMDs of a meromorphic matrix. We have generalized important results which hold for MFDs and PMDs of a rational matrix to  MFDs and AMDs of a meromorphic matrix. We have analyzed system equivalence, AMDs of least order and their characterization by their associated transfer functions. We have presented canonical forms of transfer functions, AMDs and state matrices and established relationship between them. We have also investigated  zeros and poles including structural indices of transfer functions and AMDs. The framework so presented can be utilized to analyze LTI dynamical systems such as LTI time-delay systems whose system matrix is holomorphic and transfer function is a (transcendental) meromorphic matrix.

{\bf Acknowledgement:} The research work of the second author was funded by IIT Guwahati in the form of a research fellowship.

\end{document}